%% file: main.tex
\documentclass[a4paper,12pt]{article}

\input{header}

\newcommand{\fVert}{\cV}
\newcommand{\fDiag}{\cD}
\newcommand{\fHorz}{\cH}
\newcommand{\face}{\delta}
\newcommand{\osum}{\oplus}

\newcommand{\mosum}{\mathbin{\check{\oplus}}}
\newcommand{\vpro}{\vee}
\newcommand{\ffib}[2]{#1^{#2}}
\newcommand{\ftop}[1]{\bar{#1}}
\newcommand{\fbot}[1]{\underaccent{\bar}{#1}}
\newcommand{\vface}[1]{\mathsf{v}^{#1}}
\newcommand{\sface}[1]{\mathsf{d}_{#1}}
\newcommand{\dface}[1]{\mathsf{s}_{#1}}

\newcommand{\bbface}[2]{\mathsf{w}_{#1}^{#2}}

\newcommand{\vbound}[2]{\partial^{#2}#1}
\newcommand{\sbound}[2]{\partial_{#2}#1}
\newcommand{\dbound}[2]{\mathfrak{d}_{#2}#1}
\newcommand{\aint}[2]{#2\lceil #1\rceil}
\newcommand{\ip}[2]{\varsigma^{0}_{#2#1}}
\newcommand{\ep}[2]{\varsigma^{1}_{#2#1}}
\newcommand{\bv}[2]{\varsigma^{#1}_{#2}}

\newcommand{\fn}{\boldsymbol{\eta}}

\newcommand{\fm}{\boldsymbol{\mu}}
\newcommand{\fk}{\boldsymbol{\kappa}}
\newcommand{\dvp}{\dot{p}}

\newcommand*{\rev}[1]{
  \protect\leavevmode
  \begingroup
    \color{red}#1%
  \endgroup
    }
\newcommand{\mlgar}{\rev{\longrightarrow}}

\begin{document}
\maketitle
\vspace{-1cm}
\input{abstract}

\setcounter{tocdepth}{2}
\tableofcontents
\newpage

\input{sec_intro}
\input{sec_fatDelta}
\input{sec_proof.tex}

\printbibliography
\end{document}

%% file: header.tex
\usepackage[english]{babel}
\usepackage{csquotes}
\usepackage[dvipsnames]{xcolor} 
\usepackage{amsmath,amssymb,amsthm,mathtools} 
\usepackage{extpfeil} 
\usepackage{xifthen} 
\usepackage{pict2e} 
\usepackage{enumerate} 
\usepackage{xspace} 

\usepackage{authblk} 

\renewcommand\Affilfont{\normalsize}
\makeatletter
\renewcommand\AB@affilsepx{, \protect\Affilfont}
\makeatother

\usepackage[T1]{fontenc}
\usepackage{baskervillef}
\usepackage[varqu,varl,var0]{inconsolata}
\usepackage[scale=.95,type1]{cabin}
\usepackage[libertine,vvarbb]{newtxmath}
\usepackage[cal=boondoxo]{mathalfa}

\usepackage{accents}
\usepackage{microtype} 

\DeclareFontFamily{U}{min}{}
\DeclareFontShape{U}{min}{m}{n}{<-> udmj30}{}

\DeclareSymbolFont{yhlargesymbols}{OMX}{yhex}{m}{n}
\DeclareMathAccent{\yhwidehat}{\mathord}{yhlargesymbols}{"62}

\usepackage[
  colorlinks=true,  
  linktocpage=true, 
  linkcolor=NavyBlue,    
  citecolor=Green,  
  urlcolor=BrickRed, 
]{hyperref} 
\usepackage[capitalize]{cleveref} 
\usepackage{orcidlink} 

\usepackage[
  style=alphabetic,
  maxbibnames=6,
  backend=biber,
  maxalphanames=5
]{biblatex} 
\DeclareFieldFormat{postnote}{#1}
\DeclareFieldFormat{multipostnote}{#1}
\usepackage{geometry} 
\usepackage{fancyhdr}
\usepackage{tocloft}
\cftsetindents{section}{0em}{1.5em}
\cftsetindents{subsection}{0em}{2em}
\usepackage{titlesec} 
\titleformat*{\section}{\large\scshape\center}
\titleformat*{\subsection}{\normalsize\scshape}
\titleformat*{\subsubsection}{\normalsize}

\usepackage{tikz-cd}
\usetikzlibrary{
  backgrounds, 
  decorations.markings,
  decorations.pathreplacing,
  shapes.geometric,
  matrix,arrows,
  chains,positioning,
  scopes
}

\pgfdeclarearrow{
  name = pxto,
  setup code = {
    \pgfarrowssettipend{1.5\pgflinewidth}
    \pgfarrowssetbackend{-2.5508\pgflinewidth}
    \pgfarrowssetlineend{-.25\pgflinewidth}
    \pgfarrowssetvisualbackend{-0.021\pgflinewidth}
    \pgfarrowsupperhullpoint{1.5\pgflinewidth}{0\pgflinewidth}
    \pgfarrowsupperhullpoint{-2.0085\pgflinewidth}{3.6525\pgflinewidth}
    \pgfarrowsupperhullpoint{-2.5508\pgflinewidth}{3.0763\pgflinewidth}
  },
  drawing code = {
    \pgfsetdash{}{0pt}%
    \pgfpathmoveto{\pgfpoint{1.5\pgflinewidth}{0.0254\pgflinewidth}}%
    \pgfpathlineto{\pgfpoint{-2.0085\pgflinewidth}{3.6525\pgflinewidth}}%
    \pgfpathlineto{\pgfpoint{-2.5508\pgflinewidth}{3.0763\pgflinewidth}}%
    \pgfpathlineto{\pgfpoint{-0.4322\pgflinewidth}{0.5\pgflinewidth}}%
    \pgfpathlineto{\pgfpoint{-0.4322\pgflinewidth}{-0.5\pgflinewidth}}%
    \pgfpathlineto{\pgfpoint{-2.5508\pgflinewidth}{-3.0763\pgflinewidth}}%
    \pgfpathlineto{\pgfpoint{-2.0085\pgflinewidth}{-3.6525\pgflinewidth}}%
    \pgfpathclose%
    \pgfusepathqfill}}

\tikzset{>=pxto}
\tikzcdset{arrow style=tikz}
\tikzset{mid vert/.style={
  /utils/exec=\tikzset{every node/.append style={outer sep=0.8ex}},
  postaction=decorate,
  decoration={
    markings,
    mark=at position 0.5 with {\draw[solid,-] (0,#1) -- (0,-#1);}}},
  mid vert/.default=0.75ex}
\tikzset{emptyNode/.style={shape=coordinate}} 

\DeclareMathOperator{\RelSemiCat}{\mathsf{RelSCat}}
\DeclareMathOperator{\Iso}{\mathsf{Iso}}

\DeclareMathOperator{\Mono}{\mathsf{Mono}}

\makeatletter

\let\De\Delta

\newcommand\bDe{\boldsymbol{\De}}

\let\ea\expandafter
\newcount\foreachcount

\def\foreachletter#1#2#3{\foreachcount=#1
  \ea\loop\ea\ea\ea#3\@alph\foreachcount
  \advance\foreachcount by 1
  \ifnum\foreachcount<#2\repeat}

\def\foreachLetter#1#2#3{\foreachcount=#1
  \ea\loop\ea\ea\ea#3\@Alph\foreachcount
  \advance\foreachcount by 1
  \ifnum\foreachcount<#2\repeat}

\def\definescr#1{\ea\gdef\csname s#1\endcsname{\ensuremath{\mathscr{#1}}\xspace}}
\foreachLetter{1}{27}{\definescr}

\def\definecal#1{\ea\gdef\csname c#1\endcsname{\ensuremath{\mathcal{#1}}\xspace}}
\foreachLetter{1}{27}{\definecal}

\def\definebold#1{\ea\gdef\csname b#1\endcsname{\ensuremath{\mathbf{#1}}\xspace}}
\foreachLetter{1}{27}{\definebold}

\def\definebb#1{\ea\gdef\csname d#1\endcsname{\ensuremath{\mathbb{#1}}\xspace}}
\foreachLetter{1}{27}{\definebb}

\def\definefrak#1{\ea\gdef\csname fr#1\endcsname{\ensuremath{\mathfrak{#1}}\xspace}}
\foreachletter{1}{27}{\definefrak}
\foreachLetter{1}{27}{\definefrak}

\def\definesf#1{\ea\gdef\csname sf#1\endcsname{\ensuremath{\mathsf{#1}}\xspace}}
\foreachletter{1}{27}{\definesf}
\foreachLetter{1}{27}{\definesf}

\def\definebar#1{\ea\gdef\csname #1bar\endcsname{\ensuremath{\overline{#1}}\xspace}}
\foreachLetter{1}{27}{\definebar}
\foreachletter{1}{8}{\definebar} 
\foreachletter{9}{15}{\definebar} 
\foreachletter{16}{27}{\definebar}

\def\definetil#1{\ea\gdef\csname #1til\endcsname{\ensuremath{\widetilde{#1}}\xspace}}
\foreachLetter{1}{27}{\definetil}
\foreachletter{1}{27}{\definetil}

\def\definehat#1{\ea\gdef\csname #1hat\endcsname{\ensuremath{\widehat{#1}}\xspace}}
\foreachLetter{1}{27}{\definehat}
\foreachletter{1}{27}{\definehat}

\def\definechk#1{\ea\gdef\csname #1chk\endcsname{\ensuremath{\widecheck{#1}}\xspace}}
\foreachLetter{1}{27}{\definechk}
\foreachletter{1}{27}{\definechk}

\def\defineul#1{\ea\gdef\csname u#1\endcsname{\ensuremath{\underline{#1}}\xspace}}
\foreachLetter{1}{27}{\defineul}
\foreachletter{1}{27}{\defineul}

\def\rightarrowtailfill@{\arrowfill@{\Yright\joinrel\relbar}\relbar\rightarrow}
\newcommand\xrightarrowtail[2][]{\ext@arrow 0055{\rightarrowtailfill@}{#1}{#2}}

\newcommand{\vbto}{}
\newcommand{\vblongto}{}

\DeclareRobustCommand{\vbto}{\mathrel{\mathpalette\p@to@gets\to}}
\DeclareRobustCommand{\vblongto}{\mathrel{\mathpalette\p@longto@gets\longrightarrow}}
\newcommand{\p@to@gets}[2]{\ooalign{\hidewidth$\m@th#1\mapstochar\mkern5mu$\hidewidth\cr$\m@th#1\to$\cr}}
\newcommand{\p@longto@gets}[2]{\ooalign{\hidewidth$\m@th#1\mapstochar\mkern5mu$\hidewidth\cr$\m@th#1\longrightarrow$\cr}}

\newcommand{\leqnomode}{\tagsleft@true\let\veqno\@@leqno} 

\newcommand{\adjunction}{\@ifstar\named@adjunction\normal@adjunction}
\newcommand{\normal@adjunction}[4]{%
  #1\colon #2%
  \mathrel{\vcenter{%
    \offinterlineskip\m@th
    \ialign{%
      \hfil$##$\hfil\cr
      \longrightharpoonup\cr
      \noalign{\kern-.3ex}
      \smallbot\cr
      \noalign{\kern.2ex}
      \longleftharpoondown\cr}}}%
  #3 \noloc #4}
\newcommand{\named@adjunction}[4]{%
  #2%
  \mathrel{\vcenter{%
    \offinterlineskip\m@th
    \ialign{%
      \hfil$##$\hfil\cr
      \scriptstyle#1\cr
      \noalign{\kern.1ex}
      \longrightharpoonup\cr
      \noalign{\kern-.3ex}
      \smallbot\cr
      \longleftharpoondown\cr
      \scriptstyle#4\cr}}}%
  #3}

\newcommand{\longrightharpoonup}{\relbar\joinrel\rightharpoonup}
\newcommand{\longleftharpoondown}{\leftharpoondown\joinrel\relbar}
\newcommand\noloc{%
  \nobreak
  \mspace{6mu plus 1mu}
  {:}
  \nonscript\mkern-\thinmuskip
  \mathpunct{}
  \mspace{2mu}}
\newcommand{\smallbot}{%
  \begingroup\setlength\unitlength{.25em}%
  \begin{picture}(1,1)
  \roundcap
  \polyline(0,0)(1,0)
  \polyline(0.5,0)(0.5,1)
  \end{picture}%
  \endgroup}
\makeatother

\newcommand{\fDel}{\underline{\bDe}}

\newcommand{\id}{\mathsf{id}}

\newcommand{\fat}[1]{\underline{#1}}
\newcommand{\Arr}[1]{#1^{\to}}
\newcommand{\blank}{\mathord{\hspace{0.25ex}\underline{\hspace{1ex}}\hspace{0.25ex}}}
\renewcommand{\:}{\colon}

\usepackage{thmtools}
\usepackage{thm-restate} 

\newtheoremstyle{TheoremNum}
{\topsep}{\topsep}              
{\itshape}                      
{}                              
{\bfseries}                     
{.}                             
{ }                             
{\thmname{#1}\thmnote{ \bfseries #3}}

\theoremstyle{TheoremNum}
\newtheorem{thmn}{Theorem}

\theoremstyle{plain}
\newtheorem{thm}{Theorem}[section]
  \Crefname{thm}{Theorem}{Theorems}
\newtheorem{prop}[thm]{Proposition}
  \Crefname{prop}{Proposition}{Propositions}
\newtheorem{cor}[thm]{Corollary}
  \Crefname{cor}{Corollary}{Corollaries}
\newtheorem{lem}[thm]{Lemma}
  \Crefname{lem}{Lemma}{Lemmas}

\theoremstyle{definition}
\newtheorem{dfn}[thm]{Definition}
  \Crefname{dfn}{Definition}{Definitions}
\declaretheorem[sibling=thm,style=definition,title={},qed=\rotatebox{90}{\ensuremath\lfloor}]{para}
  \Crefname{para}{\S\!}{\S\!}

\theoremstyle{remark}
\newtheorem{rmk}[thm]{Remark}
  \Crefname{rmk}{Remark}{Remarks}

\title{
		On combinatorial aspects of fat Delta}
\author{Sti\'{e}phen Pradal \orcidlink{0009-0003-9733-5228}}
\affil{School of Computer Science\\
       University of Nottingham}
\date{\vspace{-1em}
		\normalsize\today}

%% file: abstract.tex
\abstract{
The category fat Delta, introduced by J.\ Kock, is a modification of the simplex category where the degeneracies behave weakly.
The objective of this note is to provide tools for working with fat Delta.
In particular, we identify three types of morphisms: degenerated, standard and vertical faces, and establish six relations between these classes. We then show that fat Delta is generated by these morphisms and relations.
}


%% file: sec_intro.tex
\addtocontents{toc}{\protect\setcounter{tocdepth}{1}}
\section{Introduction}\label{sec:intro}
The fat Delta category, denoted by $\fDel$, was originally introduced by J.\ Kock \cite{Kock2006} in the context of Simpson's conjecture as a tool for studying weak units in higher categories.
Independently, a variant (with the same objects but fewer morphisms) was introduced by C.\ Sattler and N.\ Kraus \cite{KS2017} in the context of type theory, aiming to replace the simplex category $\Delta$ by a direct category (a category whose morphisms go only in one direction).
Despite their differing motivations, both approach agreed on the need of modifying the degeneracies of $\Delta$ in order to provide a diagrammatic interpretation of weak identity structures in higher category theory.

This paper introduces several tools for working with $\fDel$ and provides a number of essential properties of this category, with the goal of facilitating future use of it.
Although $\fDel$ admits different descriptions \cite{Kock2006}, its objects and morphisms can prove challenging to work with, as it is not as well understood as the simplex category $\Delta$.

Our work establishes a presentation of $\fDel$ in terms of generators and relations.
A key feature of the simplex category is its presentation by faces and degeneracies.
Given that $\fDel$ is a modification of $\Delta$, in the sense that the degeneracies behave weakly in $\fDel$, they share many similarities.
In particular, the simplicial faces, degeneracies and relations can be lifted to $\fDel$, resulting in three classes of faces and six relations.

\subsection{Main results}
In this paper, we study the category $\fDel$ as presented by generators and relations (\cref{sec:genrel}). We begin by distinguishing three types of morphisms, represented by specific squares in $\Delta$, which we call degenerated, vertical and standard faces, respectively:
\begin{equation*}
	\begin{tikzcd}
		{[m]}
			\ar[r, equal]
			\ar[d, "\fn"', twoheadrightarrow]
			\ar[dr,"\dface{i}" description, phantom] &
		{[m]}
			\ar[d, "\dbound{\fn}{i}", twoheadrightarrow] \\
		{[n]}
			\ar[r, "\sigma_i"', twoheadrightarrow] &
		{[n-1]}
	\end{tikzcd}
	\quad \text{and}\quad
	\begin{tikzcd}
		{[m-1]}
			\ar[r, "\face_{i}", mid vert, hook]
			\ar[d, "\vbound{\fn}{i}"', twoheadrightarrow]
			\ar[dr,"\vface{i}" description, "\ulcorner" very near end, phantom] &
		{[m]}
			\ar[d, "\fn", twoheadrightarrow] \\
		{[n]}
			\ar[r, equal] &
		{[n]}
	\end{tikzcd}
	\quad \text{and}\quad
	\begin{tikzcd}
		{[m-1]}
			\ar[r, "\face_{\bv{}{i}}",hook]
			\ar[d, "\sbound{\fn}{i}"', twoheadrightarrow]
			\ar[dr,"\sface{i}" description, phantom] &
		{[m]}
			\ar[d, "\fn", twoheadrightarrow] \\
		{[n-1]}
			\ar[r, "\face_i"',hook] &
		{[n]}
	\end{tikzcd}
	\end{equation*}

Each of these morphisms plays a unique role in the structure of $\fDel$.
Vertical faces are tied to active maps in $\Delta$, denoted by the arrow $\vblongto$, and act by splitting a marked edge into two by inserting an inner marked vertex.
Standard faces, on the other hand, split an unmarked edge into two by inserting a standard vertex $\bv{}{i}$.
More generally, the natural number $\ep{}{i}$ (resp.\ $\ip{}{i}$) denotes the maximal (resp.\ minimal) vertex in the fibre of $\fn$ over $i$.
Degenerated faces stand apart as the only faces capable of turning unmarked edges into marked ones.
A special composition, denoted by $\bbface{i}{\epsilon}$, combines degenerated and standard faces as $\dface{i}\sface{i+\epsilon}$, where $\epsilon$ takes the value $0$ or $1$.
All the notation and terminology helping us navigate the intricate structure of $\fDel$ is made precise throughout the paper.

These three classes of faces obey the following six relations:

\noindent
\begin{minipage}{.5\linewidth}
\begin{align}
	\dface{j}\dface{i} &= \dface{i}\dface{j+1} \hspace{0.825cm} i\leq j\leq n, \tag{\ref{rel:dd}}\\
	\sface{i}\sface{j} &= \sface{j+1}\sface{i} \hspace{0.725cm} i\leq j\leq n, \tag{\ref{rel:hh}}\\
	\vface{i}\vface{j} &= \vface{j+1}\vface{i} \hspace{0.725cm} i\leq j\leq m, \tag{\ref{rel:vv}}
\end{align}	
\end{minipage}\smallskip
\noindent
\begin{minipage}{.5\linewidth}
\begin{align}
	\sface{i}\vface{j} &=
	\begin{cases}
		\vface{j+1}\sface{i} & \bv{}{i} < j \leq m \\
		\vface{j}\sface{i}   & j < \bv{}{i}
	\end{cases} \tag{\ref{rel:hv}}\\
	\vface{j}\dface{i} &=
	\begin{cases}
		\dface{i}\bbface{i+1-\epsilon}{\epsilon} & j=\bv{}{i+1} \\
		\dface{i}\vface{j}   & j \neq \bv{}{i+1}
	\end{cases} \tag{\ref{rel:dv}}\\
	\dface{j}\sface{i} &=
	\begin{cases}
		\sface{i}\dface{j-1} & i<j\leq n \\
		\bbface{j}{\epsilon}	   & i=j+\epsilon \\
		\sface{i-1}\dface{j} & j+1<i \leq n
	\end{cases}\tag{\ref{rel:hd}}
\end{align}	
\end{minipage}\smallskip

\noindent
Finally, we prove that any morphism in $\fDel$ factors uniquely, up to these relations, as a composition of degenerated, vertical and standard faces.

\begin{thmn}[\ref{thm:grdu}]
	Any map in $\fDel$ factors uniquely up to relations \eqref{rel:dd}--\eqref{rel:hd} as a composition of degenerated, vertical and standard faces.
\end{thmn}
As a consequence, we recover the ternary factorisation system sketched in \cite{KockCom,Sattler2017} as \cref{cor:tfs}.

We believe that this presentation, which is a non-trivial result to establish, has potential for future applications of $\fDel$, in the same way that simplicial identities have been a fundamental tool for the use of simplicial methods in homotopy and higher category theory.

\subsection{Acknowledgements}
I am grateful to Nicolai Kraus and Tom de Jong for their valuable support and guidance throughout the development of this work.
I also wish to thank Simona Paoli and Stefania Damato for their insightful comments and suggestions on an early draft, as well as Josh Chen and Mark Williams for inspiring discussions.
This work was supported by grants from the Royal Society (URF\textbackslash{}R1\textbackslash{}191055).

\addtocontents{toc}{\protect\setcounter{tocdepth}{2}}


%% file: sec_fatDelta.tex
\section{Preliminaries on fat Delta}\label{sec:fDeldef}
The category $\fDel$ was first introduced by J.\ Kock in \cite{Kock2006} through the concept of  ``coloured semiordinals'', but it can be viewed in various other ways, as mentioned by J.\ Kock in the same work and proved in \cite{JKPP2025}.
In this section, we review two descriptions introduced in \cite{Kock2006}.
First, we give a brief outline of the original description, which is useful for intuition, in terms of relative semiordinals (\cref{sec:rso}).
Next, we present a more practical description, also presented in \cite{KS2017}, in terms of a subcategory of the category of epimorphisms in $\Delta$ (\cref{sec:epifdel}). While we will use the intuition from the relative semiordinals, our work exclusively relies on the definition from \cref{sec:epifdel}.

The current section also introduces tools and notation to simplify the manipulation of the objects of $\fDel$ (\cref{sec:nota}), and the active-inert factorisation system from \cite{JKPP2025} (\cref{sec:activeinert}).

\subsection{Relative semiordinals}\label{sec:rso}
In \cite{Kock2006}, motivated by a means to give a diagrammatic interpretation of weak identity arrows in higher categories, J.\ Kock defines $\fDel$ conceptually as the full subcategory of the category of relative semicategories, $\RelSemiCat$. Relative semicategories consists of non-unital categories equipped with a wide subsemicategory. We present this description here to provide an intuition for the objects and morphisms of $\fDel$.

\begin{para}
A \emph{semiordinal} is a semicategory associated with a total strict order relation.
In particular, the relation is not reflexive, which means that there are no identity morphisms.
\end{para}

To illustrate this, we can view finite semiordinals as strings of arrows as follows:
\begin{equation}\label{diag:semord}
	0, \quad 0 \longrightarrow 1, \quad 0 \longrightarrow 1 \longrightarrow 2, \quad 0 \longrightarrow 1 \longrightarrow 2 \longrightarrow 3, \quad \ldots
\end{equation}
The category of finite non-empty semiordinals is naturally identified with the semisimplex category $\Delta_+$ (i.e.\ the simplex category $\Delta$ restricted to its monomorphisms), as the absence of identities allows only monomorphisms between semiordinals.

\begin{para}
	\emph{Relative semiordinal} are defined as free relative semicategories on linearly ordered relative graphs.
	The category $\fDel$, fat Delta, is the category of finite non-empty relative semiordinals and marking-preserving (i.e.\ colour-preserving in~\cite{Kock2006}) functors between them.
\end{para}

This description induces a fully faithful functor $\fDel \hookrightarrow \RelSemiCat$. Therefore, objects in $\fDel$ can be visualised as strings of marked and unmarked arrows. This includes objects as in \eqref{diag:semord}, as well as objects with some of the edges possibly marked (in red), such as
\begin{equation}\label{diag:lrneut}
	0 \mlgar 1 \longrightarrow 2 \quad \text{or} \quad 0 \longrightarrow 1 \mlgar 2,
\end{equation}
and fully marked semiordinals
\begin{equation}
	0 \mlgar 1, \quad 0 \mlgar 1 \mlgar 2, \quad 0 \mlgar 1 \mlgar 2 \mlgar 3, \quad \ldots
\end{equation}
Morphisms between these objects are strictly monotone injections preserving the marking, i.e.\ they map red arrows to red arrows.
Observe that unmarked edges can also be mapped to marked ones.
The motivation for introducing this marking is that the marked edges can be interpreted as (weak) identities.

This approach offers a very intuitive way of thinking about $\fDel$, however it is not the most practical for calculations.

\subsection{Epimorphisms in \texorpdfstring{$\Delta$}{Delta}}\label{sec:epifdel}
In this section, we give a different, but much more practical, perspective on $\fDel$ which we recall from \cite{Kock2006}.
Although the two descriptions are isomorphic, as shown in \cite[Theorem~4.22]{JKPP2025}, this is the definition we will use throughout the rest of the paper.
Henceforth, the notation $\fDel$ will exclusively refer to the definition presented in this section.

As mentioned above, the purpose of $\fDel$ is to provide simplicial-like methods for higher structures while offering a weak treatment of the degeneracies in~$\Delta$.
Therefore, one of the potential directions taken to construct this category consists of expanding (or ``fattening'') the simplex category $\Delta$ with all the possible degeneracies as objects.
This approach yields~$\fDel$ and enables diagrammatic access to the identity structures, similar to that of the composition structure.

\begin{para}
	The category $\fDel$ is defined as $(\Arr{\Delta}_{-})_{\Mono}$, the subcategory of monomorphisms in the arrow category of epimorphisms $\Delta_-$ of $\Delta$.
	In other words, the objects of $\fDel$ are the epimorphisms $\fn\:[m] \twoheadrightarrow [n]$ in $\Delta$, i.e.\ compositions of degeneracies, and the morphisms $f\:\fn \to \fk$ are commutative squares in $\Delta$ of the form
	\begin{equation}\label{diag:mofDel}
	\begin{tikzcd}
		{[m]}
			\ar[r, "\ftop{f}", hook]
			\ar[d, "\fn"', twoheadrightarrow]
		& {[n]}
			\ar[d, "\fk", twoheadrightarrow] \\
		{[k]}
			\ar[r, "\fbot{f}"']
		& {[l]}
	\end{tikzcd}
	\end{equation}
	whose vertical arrows are epimorphisms and horizontal top arrow is a monomorphism.
\end{para}

\begin{para}
	Let $\fn\:[m] \twoheadrightarrow [n]$ be an object of $\fDel$.
	By \emph{vertex/edge} of $\fn$ we mean a vertex/edge of the domain $[m]$, and we say that an edge is \emph{marked} if it is sent to an identity by~$\fn$.
\end{para}

In terms of \cref{sec:rso}, given an object $\fn$ in $\fDel$, the source of the epimorphism represents the underlying semiordinal of the respective relative semiordinal, and the epimorphism encodes the marking by collapsing the marked edges of the domain to the identities of the codomain.
To illustrate this point, consider the objects depicted in \eqref{diag:lrneut}, which can respectively be diagrammatically represented by the epimorphisms $\sigma_0,\sigma_1\:[2] \twoheadrightarrow [1]$ as
\begin{equation*}
\begin{tikzcd}[column sep=0.2cm]
	0
		\ar[dr,mapsto]
		\ar[rr, dash, red]
	&& 1
		\ar[dl,mapsto]
		\ar[rr, dash]
	&& 2
		\ar[d,mapsto] \\
	& 0
		\ar[rrr, dash]
	&&& 1
\end{tikzcd}
\quad\text{and}\quad
\begin{tikzcd}[column sep=0.2cm]
	0
		\ar[d,mapsto]
		\ar[rr, dash]
	&& 1
		\ar[dr,mapsto]
		\ar[rr, dash, red]
	&& 2
		\ar[dl,mapsto] \\
	0
		\ar[rrr, dash]
	&&& 1
	&
\end{tikzcd}
\end{equation*}
respectively, and where a marked edge is collapsed to a single vertex.

\begin{para}
	In a similar manner, we define the augmented fat Delta category $\fDel_{\ast}$ by restricting the augmented simplex category $\Delta_{\ast}$ to its epimorphisms. 
\end{para}

Alternatively, $\fDel_{\ast}$ can be obtained by formally adding an initial object to $\fDel$ or by taking the category of non-empty relative semiordinals.

\subsection{Functors on \texorpdfstring{$\fDel$}{fat Delta}}
The category $\fDel$ shares deep relations with the simplex category $\Delta$ and the semisimplex category $\Delta_+$. In particular, $\Delta_+$ can be embedded in $\fDel$ in two ways and $\fDel$ has been shown to be a direct replacement of $\Delta$ in \cite{Sattler2017}.

\begin{para}
	The restriction of the domain and codomain functors provides projections denoted by
	\begin{equation*}
		\overline{\pi}\:\fDel \to \Delta_+ \quad \text{and} \quad \fat{\pi}\:\fDel \to \Delta,
	\end{equation*}
respectively.\qedhere
\end{para}

The functor $\fat{\pi}$ is significant as it induces a Dwyer-Kan equivalence between the homotopical categories $(\Delta,\Iso)$ and $(\fDel,\fVert)$, where $\fVert$ is the class of morphisms mapped to identities in $\Delta$ by $\fat{\pi}$ \cite{Sattler2017}.
In addition, $\fat{\pi}$ is an \emph{ambifibration}, which endows it with convenient lifting properties (see \cref{para:tfact}).
The literature on ambifibrations is quite limited; for more details, we refer the reader to the unpublished note \cite{Sattler2017} or J.\ Kock's comment on $n$-Category Caf\'{e} \cite{KockCom}.

We do not make use of the following lemma in this paper, but we think that it is worth observing.

\begin{lem}
    The projection $\overline{\pi}\: \fDel \to \Delta_+$ is a Grothendieck opfibration.
\end{lem}
\begin{proof}
	Let $\fn\:[m] \twoheadrightarrow [n]$ be an object of $\fDel$ and $\ftop{f}\:[m] \hookrightarrow [k]$. By \autocite[Corollary 3.3]{CFPS2023}, the span $[n] \twoheadleftarrow [m] \hookrightarrow [k]$ has a pushout and hence we get a map $f\:\fn \to \fk$ over $\ftop{f}$ with respect to $\overline{\pi}$:
    \begin{equation*}
    \begin{tikzcd}
      {[m]}
	  	\ar[d, "\fn"', twoheadrightarrow]
		\ar[r, "\ftop{f}", hookrightarrow]
		\ar[dr, "\ulcorner" description, phantom, very near end]
      & {[k]}
	  	\ar[d, "\fk", twoheadrightarrow] \\
      {[n]}
	  	\ar[r, "\fbot{f}"']
      & {[l]}
    \end{tikzcd}
    \end{equation*}
    It remains to show that $f\:\fn \to \fk$ is a $\overline{\pi}$-cocartesian lift of $\ftop{f}$.
    For any $\ftop{g}\:[k] \hookrightarrow [a]$, $\fm\:[a] \twoheadrightarrow [b]$ and $h\:\fn \to \fm$ such that $\ftop{h}=\ftop{g}\ftop{f}$, the universal property of $[l]$ yields a unique morphism $\fbot{g}\:[l] \to [b]$ making the diagrams commute. In particular, we have a unique map $g\:\fk \to \fm$ such that $gf=h$.
\end{proof}

The category $\fDel$ comes equipped with two natural inclusions from $\Delta_+$ \cite{Kock2006}, which we now describe.

\begin{para}
	The functors denoted by
	\begin{equation*}
		(\blank)^{\flat}\:\Delta_+ \hookrightarrow \fDel
		\quad \text{and} \quad
		(\blank)^{\sharp}\:\Delta_+ \hookrightarrow \fDel
	\end{equation*}
	are respectively called the \emph{horizontal} and \emph{vertical inclusions}.
	The horizontal inclusion $(\blank)^{\flat}$ sends an object $[m]$ to the identity morphism $\id_{[m]}$ and a monomorphism $[m] \hookrightarrow [m']$ to the square
	\begin{equation*}
	\begin{tikzcd}
		{[m]}
			\ar[r, hook]
			\ar[d, equal]
		& {[m']}
			\ar[d, equal] \\
		{[m]}
			\ar[r, hook]
		& {[m']}
	\end{tikzcd}
	\end{equation*}
	The vertical inclusion $(\blank)^{\sharp}$ sends an object $[m]$ to the canonical map $[m] \twoheadrightarrow [0]$ and a monomorphism $[m] \hookrightarrow [m']$ to the square
	\belowdisplayskip=-5pt
	\begin{equation*}
	\begin{tikzcd}
		{[m]}
			\ar[r, hook]
			\ar[d, twoheadrightarrow]
		& {[m']}
			\ar[d, twoheadrightarrow] \\
		{[0]}
			\ar[r, equal]
		& {[0]} 
	\end{tikzcd}
	\end{equation*}
\end{para}

The objects in the image of the horizontal (resp.\ vertical) inclusion intuitively correspond to the relative semiordinals with nothing marked (resp.\ everything marked). These two inclusions interact nicely with the projection~$\overline{\pi}$.

\begin{prop}
	The horizontal and vertical inclusions $(\blank)^{\flat}$ and $(\blank)^{\sharp}$ are respectively left and right adjoint to the projection $\overline{\pi}$:
	\begin{equation*}
	\begin{tikzcd}
		\fDel
			\ar[rr, "\overline{\pi}"{name = mid, description}]
		&& \Delta_+
			\ar[ll, "(\blank)^{\flat}"{name = bot, description}, hook, bend left=60]
			\ar[ll, "(\blank)^{\sharp}"{name = top, description}, hook', bend right=60]
		\ar[phantom, from=mid, to=bot, "\dashv" rotate=90]
		\ar[phantom, from=mid, to=top, "\dashv" rotate=90]
	\end{tikzcd}		
	\end{equation*}
\end{prop}
\begin{proof}
	Let $[k]$ and $\fn\: [m] \twoheadrightarrow [n]$ be objects of $\Delta_+$ and $\fDel$, respectively.
	To prove the adjunction ${(\blank)^{\sharp} \vdash \overline{\pi}}$, we define the counit ${\epsilon_{[k]} \: \overline{\pi}([k]^{\sharp}) \to [k]}$ to be the identity $\id_{[k]}$ and the unit ${\nu_{\fn}\:\fn \to \overline{\pi}(\fn)^{\sharp}}$ to be the left square in \eqref{diag:(co)unit} below.
	Similarly, to prove the adjunction ${\overline{\pi} \vdash (\blank)^{\flat}}$, we define the unit ${\nu'_{[k]}\: [k] \to \overline{\pi}([k]^{\flat})}$ to be the identity $\id_{[k]}$ and the counit ${\epsilon'_{\fn}\: \overline{\pi}(\fn)^{\flat} \to \fn}$ to be the right square in \eqref{diag:(co)unit}.
	\begin{equation}\label{diag:(co)unit}
	\begin{tikzcd}
		{[m]}
			\ar[r, equal]
			\ar[d, "\fn"', twoheadrightarrow]
			\ar[dr, "\nu_{\fn}" description, phantom]
		& {[m]}
		\ar[d, "\overline{\pi}(\fn)^{\sharp}", twoheadrightarrow] \\
		{[n]}
			\ar[r, twoheadrightarrow]
		& {[0]}
	\end{tikzcd}
	\qquad\qquad
	\begin{tikzcd}
		{[m]}
			\ar[r, equal]
			\ar[d, "\overline{\pi}(\fn)^{\flat}"', equal]
			\ar[dr, "\epsilon'_{[m]}" description, phantom]
		& {[m]}
			\ar[d, "\fn", twoheadrightarrow] \\
		{[m]}
			\ar[r, "\fn"', twoheadrightarrow]
		& {[n]}
	\end{tikzcd}
	\end{equation}
	This clearly defines natural transformations, and the triangle are trivially verified as everything reduces to identities.
\end{proof}

\begin{rmk}
	The composite
	\begin{equation*}
		\Delta_+ \xhookrightarrow{(\blank)^{\flat}} \fDel \xtwoheadrightarrow{\fat{\pi}} \Delta
	\end{equation*}
	factors the canonical inclusion $\Delta_+ \hookrightarrow \Delta$.
	It is in this sense that one may regard $\fDel$ as a direct replacement of $\Delta$.
	We refer the reader to \cite{KS2017,Sattler2017} for more about this property.
\end{rmk}

\subsection{Factorisation systems on \texorpdfstring{$\fDel$}{fat Delta}}\label{sec:activeinert}
As shown in \cite{JKPP2025}, the category $\fDel$ possesses an active-inert factorisation system $(\fDel_a,\fDel_0)$.
Although we do not rely on the factorisation system in this paper, the active morphisms will be useful in \cref{sec:genrel} for defining a class of generators (the vertical faces).
This type of factorisation system typically arises in the context of the \emph{Segal conditions} and \emph{nerve functors}, as it makes it possible to express the Segal conditions in terms of a restriction along the right class~\cite{BMW2012}.
In the simplicial case, active and inert morphisms have a simple characterisation.

\begin{para}
	In the category $\Delta$, a morphism $f\: [m] \to [n]$ is said to be \emph{active} if it preserves endpoints, i.e., $f(0)=0$ and $f(m)=n$.
	It is called \emph{inert} if it preserves distances, i.e., $f(i+1)=f(i)+1$.
	We denote active morphisms with a vertical bar "$\vblongto$", while inert morphisms do not have a denotation as they will not be needed in our context.
\end{para}

In the category $\fDel$, the characterisation is similar, with the difference that the maps respect the marking.

\begin{para}
	The class $\fDel_a$ of \emph{active morphisms} is given by maps of the form
	\begin{equation*}
	\begin{tikzcd}[column sep=1.2cm]
		{[m]}
			\ar[r, "\text{active}", mid vert, hookrightarrow]
			\ar[d, twoheadrightarrow]
			\ar[dr, phantom, very near end, "{\ulcorner}"]
		& {[k]}
			\ar[d, twoheadrightarrow] \\
		{[n]}
			\ar[r]
		& {[l]}
	\end{tikzcd}
	\end{equation*}
	where the top arrow is an active monomorphism in $\Delta$ and the square is a pushout.
	The class $\fDel_0$ of \emph{inert morphisms} is given by maps of the form
	\begin{equation*}
	\begin{tikzcd}[column sep=1.2cm]
		{[m]}
			\ar[r, "\text{inert}", hookrightarrow]
			\ar[d, twoheadrightarrow] &
		{[k]}
			\ar[d, twoheadrightarrow] \\
		{[n]}
			\ar[r] &
		{[l]}
	\end{tikzcd}
	\end{equation*}
	where the top arrow is inert in $\Delta$.
\end{para}

\begin{rmk}\label{rmk:mcontract}
	Observe that the bottom map of an active morphism in $\fDel$ is active, since left classes are preserved by pushouts, and is a monomorphism since monomorphisms are preserved by pushouts along epimorphisms in $\Delta$ \cite[Lemma~3.4]{JKPP2025}.
\end{rmk}

The lifting properties of the ambifibration $\fat{\pi}$ allow us to distinguish three notable classes of morphisms in $\fDel$.

\begin{para}\label{para:tfact}
	The classes $\fDiag$, $\fVert$ and $\fHorz$ correspond respectively to $\fat{\pi}$-cocartesian lifts of epimorphisms, $\fat{\pi}$-vertical morphisms and $\fat{\pi}$--cartesian lifts of monomorphisms.
	Explicitly, they are morphisms of the form
	\begin{equation*}
	\begin{tikzcd}
		{[m]}
			\ar[r, equal]
			\ar[d, twoheadrightarrow]
			\ar[dr,"\fDiag" description, phantom] &
		{[k]}
			\ar[d, twoheadrightarrow] \\
		{[n]}
			\ar[r, twoheadrightarrow] &
		{[l]}
	\end{tikzcd}
	\quad\text{and}\quad
	\begin{tikzcd}
		{[m]}
			\ar[r, hook]
			\ar[d, twoheadrightarrow]
			\ar[dr,"\fVert" description, phantom] &
		{[k]}
			\ar[d, twoheadrightarrow] \\
		{[n]}
			\ar[r, equal] &
		{[l]}
	\end{tikzcd}
	\quad \text{and}\quad
	\begin{tikzcd}
		{[m]}
			\ar[r, hook]
			\ar[d, twoheadrightarrow]
			\ar[dr, "\fHorz" description, phantom] &
		{[k]}
			\ar[d, twoheadrightarrow] \\
		{[n]}
			\ar[r, hook] &
		{[l]}
	\end{tikzcd}
	\end{equation*}
	We call them \emph{diagonal}, \emph{vertical} and \emph{horizontal} morphisms, respectively.
\end{para}

The terminology used for the classes $\fDiag$ and $\fHorz$ arises from their correspondence to the diagonal and horizontal arrows, respectively, in the diagrammatic sketch of the category $\fDel$ below.
The terminology vertical is standard in the categorical literature on Grothendieck (op)fibrations, but observe that the morphisms in $\fVert$ are also drawn vertically in the diagram.
\newcommand\x{1mm}
\newcommand\xx{2mm}
\begin{equation*}
\begin{tikzcd}[row sep=1.2cm, column sep=1cm]
		{[0]}
			\ar[rr, shift left=\x]
			\ar[rr, shift left=-\x]
			\ar[d, shift left=\x]
			\ar[d, shift left=-\x]
		&&[-7ex]
		{[1]^{\flat}}
			\ar[rr, shift left=\xx]
			\ar[rr, shift left=0mm]
			\ar[rr, shift left=-\xx]
			\ar[dll]
			\ar[dl, shift left=\x]
			\ar[dl, shift left=-\x]
			\ar[dr, shift left=\x]
			\ar[dr, shift left=-\x]
		&[-7ex] {}&
		{[2]^{\flat}}
			\ar[dlll]
			\ar[dl, bend left=10]
		&[-6ex] \cdots
		\\
		{[1]^{\sharp}}
			\ar[d, shift left=\xx]
			\ar[d, shift left=0mm]
			\ar[d, shift left=-\xx]
			\ar[r, shift left=\x]
		&
		{\sigma^{1}_{0}}
			\ar[dl]
		&&
		{\sigma^{1}_{1}}
			\ar[dlll, bend left=10]
			\ar[from=lll, shift left=0, bend right=20, crossing over]
			\ar[from=ul, shift left=\x, crossing over]
			\ar[from=ul, shift left=-\x, crossing over]
		\\
		{[2]^{\sharp}}
		\\
		[-6ex]
		\vdots
\end{tikzcd}
\end{equation*}
In addition, it follows from the properties of ambifibrations sketched in \cite{KockCom,Sattler2017} that these classes constitute a ternary factorisation system on $\fDel$.
We recover this result in \cref{cor:tfs} as an application of the decomposition of the morphisms in $\fDel$ into faces, up to relations, proved in \cref{sec:genrel}.

\subsection{Sum, relative sum and \texorpdfstring{$\vpro$}{V}-product}\label{sec:nota}
With the goal of simplifying the manipulation of the objects $\fn\:[m] \twoheadrightarrow [n]$ of $\fDel$ in mind, a tool for decomposing them is introduced in this section.
\begin{para}
	For two objects $[n]$ and $[n']$ of $\Delta$, we write $[n] \osum [n']$ for the categorical join of the two posets, and call it the ordinal sum.
As the ordinal sum is functorial, it induces a functor on the category $\fDel$, which is denoted by $\osum$ as well, and we call it the \emph{relative semiordinal sum}.
\end{para}

Intuitively, the relative semiordinal sum of two objects $\fn$ and $\fk$ in $\fDel$ connects the last vertex of $\fn$ to the first vertex of $\fk$ by an unmarked edge.
For instance, in the case of \eqref{diag:lrneut}, we have
\begin{equation*}
	(0 \longrightarrow 1 \mlgar 2) \osum (0 \mlgar 1 \longrightarrow 2) =
	0 \longrightarrow 1 \mlgar 2 \longrightarrow 3 \mlgar 4 \longrightarrow 5.
\end{equation*}

\begin{para}
	The domain of $\fn$ can be decomposed as an ordinal sum of its fibres.
	Write $\ffib{\fn}{i}$ for the natural number $\vert \fn^{-1}(i)\vert-1$, with $0 \leq i \leq n$, and $\ffib{\fn}{i}$ again for the canonical map $[\ffib{\fn}{i}] \twoheadrightarrow [0]$.
	We may thus decompose $\fn$ as follows:
	\begin{equation*}
		\ffib{\fn}{0} \osum \cdots \osum \ffib{\fn}{n}\:[\ffib{\fn}{0}] \osum \cdots \osum [\ffib{\fn}{n}] \twoheadrightarrow [n].\qedhere
	\end{equation*}
\end{para}

With this notation, the interval $[\ffib{\fn}{i}]$ of $\fn$ is marked if and only if $\ffib{\fn}{i}>0$.
Intuitively, $\ffib{\fn}{i}$ counts the number of edges (not vertices) in the fibre of $i$.

\begin{prop}
	The relative ordinal sum $\osum$ induces a strict monoidal structure on the augmented fat Delta category $\fDel_{\ast}$.
\end{prop}
\begin{proof}
	It is not difficult to check that this structure is inherited from the strict monoidal structure on $\Delta_{\ast}$.
\end{proof}

To distinguish vertices, which is crucial for \cref{sec:genrel}, we introduce the following notation.

\begin{para}\label{rmk:varsigma}
	Let $\ep{}{i}$ denote the sum $\ffib{\fn}{0} + \cdots + \ffib{\fn}{i} + i$ and let $\ip{}{i}$ denote $\ffib{\fn}{0} + \cdots + \ffib{\fn}{i-1} + i$, where the dependency on $\fn$ is implicit.
	In other words, $\ip{}{i}$ is the smallest vertex and $\ep{}{i}$ is the greatest vertex in the fibre of $\fn$ over $i$.
	In particular, we have $\ip{}{0}=0$ and $\ep{}{n}=m$ when $\fn:[m] \to [n]$.
	Finally, given a vertex $0 \leq i \leq m$ of $\fn$, we write $\aint{i}{\fn}$ or simply $\aint{i}{}$ for the associated interval $\ffib{\fn}{\fn(i)}$ in $\fn$.
\end{para}

Recall from~\cite[Definition 2.6]{CFPS2023} that the $\vpro$-product of $[n]$ and $[m]$ is defined by the pushout
\begin{equation*}
\begin{tikzcd}
	{[0]}
		\ar[r, "n"]
		\ar[d, "0"']
		\ar[dr, phantom, very near end, "{\ulcorner}"]
	& {[n]}
		\ar[d] \\
	{[m]}
		\ar[r]
	& {[n] \vpro [m]}
\end{tikzcd}
\end{equation*}
In particular, $[n] \vpro [m] =[n+m]$.

\begin{para}
	As the $\vpro$-product is functorial on active morphisms of $\Delta$, \cite[Definition 2.6]{CFPS2023}, it induces an operation on objects of $\fDel$ (epimorphisms of $\Delta$ are always active), denoted by $\vpro$ as well, and called the \emph{relative $\vpro$\nobreakdash-product}.
	Given two objects $\fn\:[m] \twoheadrightarrow [n]$ and $\fk\:[l] \twoheadrightarrow [k]$ of $\fDel$, the relative semiordinal
	\begin{equation*}
		\fn\vpro\fk\:[m]\vpro[l] \twoheadrightarrow [n]\vpro[k]
	\end{equation*}
	is obtained by gluing $\fk$ to the maximal element of $\fn$.
\end{para}

Intuitively, the relative $\vpro$-product of two objects $\fn$ and $\fk$ in $\fDel$ connects the last vertex of $\fn$ directly to the first vertex of $\fk$.
For instance, in the case of \eqref{diag:lrneut}, we have
\begin{equation*}
	(0 \longrightarrow 1 \mlgar 2) \vpro (0 \mlgar 1 \longrightarrow 2) =
	0 \longrightarrow 1 \mlgar 2 \mlgar 3 \longrightarrow 4.
\end{equation*}

\begin{para}
	It is also possible to combine the ordinal sum and the $\vpro$-product in order to define the \emph{marked semiordinal sum} $\mosum$ of two objects of $\fDel$.
	Given $\fn$ and $\fk$ as before, the marked ordinal sum is the canonical epimorphism
	\begin{equation*}
		\fn \mosum \fk\: [m] \osum [l] \twoheadrightarrow [n] \vpro [k].\qedhere
	\end{equation*}
\end{para}

Intuitively, the marked semiordinal sum is analogous to the relative semiordinal sum, with the difference that the edge linking $\fn$ and $\fk$ is marked.
For instance, in the case of \eqref{diag:lrneut}, we have
\begin{equation*}
	(0 \longrightarrow 1 \mlgar 2) \mosum (0 \mlgar 1 \longrightarrow 2) =
	0 \longrightarrow 1 \mlgar 2 \mlgar 3 \mlgar 4 \longrightarrow 5.
\end{equation*}

\begin{rmk}
	Note that we have $\fn \mosum \fk= \fn \vpro [1]^{\sharp} \vpro \fk$ and $\fn \osum \fk= \fn \vpro [1]^{\flat} \vpro \fk$.
\end{rmk}


%% file: sec_proof.tex
\section{Fat Delta by generators and relations}\label{sec:genrel}
In this section, we give a presentation of $\fDel$ in terms of generators and relations, called faces and fat simplicial relations respectively.
We believe it has potential for future applications of $\fDel$, in the same manner that simplicial identities have been a fundamental tool for the usage of simplicial methods in both homotopy and higher category theory.
First, the generators and relations are introduced and described intuitively (\cref{sec:elf,sec:frel}); then, we prove the factorisation and uniqueness up to the relations (\cref{sec:fact,sec:uniq}).
As a consequence, we recover the ternary factorisation system of \cite{KockCom,Sattler2017}.

\subsection{Faces of fat Delta}\label{sec:elf}
As seen in the previous section, the category $\fDel$ shares many similarities with $\Delta$ and $\Delta_+$.
An important feature of $\Delta$ resides in the elementary description of its morphisms in terms of simplicial faces, degeneracies and relations.
The faces and degeneracies form the Reedy factorisation system, which is then lifted to the ternary factorisation system $(\fDiag,\fVert,\fHorz)$ on $\fDel$ using the lifting properties of the ambifibration $\fat{\pi}$ \cite{KockCom,Sattler2017}.
Therefore, it is natural to expect that the relations can be passed on to $\fDel$ via $\fat{\pi}$.

Below we give a concrete definition of the generators of $\fDel$. We observe that they are given by the class of $\fat{\pi}$-cocartesian lifts of degeneracies, the class of $\overline{\pi}$-cocartesian lifts of active faces (or active $\fat{\pi}$-vertical morphisms with a simplicial face at the top), and the class of $\fat{\pi}$-cartesian lifts of faces with a simplicial face at the top.
The generators are given by
\begin{equation*}
\begin{tikzcd}
	{[m]}
		\ar[r, equal]
		\ar[d, "\fn"', twoheadrightarrow]
		\ar[dr,"\dface{i}" description, phantom] &
	{[m]}
		\ar[d, "\dbound{\fn}{i}", twoheadrightarrow] \\
	{[n]}
		\ar[r, "\sigma_i"', twoheadrightarrow] &
	{[n-1]}
\end{tikzcd}
\quad \text{and}\quad
\begin{tikzcd}
	{[m-1]}
		\ar[r, "\face_{i}", mid vert, hook]
		\ar[d, "\vbound{\fn}{i}"', twoheadrightarrow]
		\ar[dr,"\vface{i}" description, "\ulcorner" very near end, phantom] &
	{[m]}
		\ar[d, "\fn", twoheadrightarrow] \\
	{[n]}
		\ar[r, equal] &
	{[n]}
\end{tikzcd}
\quad \text{and}\quad
\begin{tikzcd}
	{[m-1]}
		\ar[r, "\face_{\bv{}{i}}",hook]
		\ar[d, "\sbound{\fn}{i}"', twoheadrightarrow]
		\ar[dr,"\sface{i}" description, phantom] &
	{[m]}
		\ar[d, "\fn", twoheadrightarrow] \\
	{[n-1]}
		\ar[r, "\face_i"',hook] &
	{[n]}
\end{tikzcd}
\end{equation*}
where the maps $\dbound{\fn}{i}$, $\vbound{\fn}{i}$ and $\sbound{\fn}{i}$ are described in detail below using the notation $\ip{}{i}$ and $\ep{}{i}$ from \cref{rmk:varsigma}.
First, we introduce some terminology for vertices that will be helpful for defining and manipulating the faces.

\begin{dfn}
	Let $i$ be a vertex of $\fn$, an object of $\fDel$. We say that $i$ is
	\begin{enumerate}[(i)]
		\item \emph{standard} if its neighbouring edges are unmarked, i.e.\ $\aint{i}{}=0$.
		\item \emph{inner marked} if $i \neq 0, n$ and both neighbouring edges are marked.
		\item \emph{left bordering} (resp. \emph{right bordering}) if $i$ has exactly one marked neighbouring edge, which is the outgoing (resp. incoming) edge.
	\end{enumerate} 
\end{dfn}

These conditions can be reformulated using the notation $\ip{}{i}$ and $\ep{}{i}$ from \cref{rmk:varsigma}.
A vertex $i$ is standard if and only if $\ip{}{\fn(i)}=i=\ep{}{\fn(i)}$, indicating that $i$ is the only vertex in its associated fibre.
This will be useful for faces and relations involving unmarked edges. We will sometimes omit the superscript $0$ or $1$ in that situation and simply write $\bv{}{\fn(i)}$ for instance.
Similarly, $i$ is inner marked if and only if $\ip{}{\fn(i)}\neq i \neq \ep{}{\fn(i)}$, a condition that will be relevant in marked contexts.
Left and right bordering vertices are those that separate standard and inner marked vertices.
They can be characterised as vertices where either $i=\ip{}{\fn(i)}$ for left bordering or $i=\ep{}{\fn(i)}$ for right bordering (but not both).
In instances where the distinction between left and right is not pertinent, these vertices are referred to as bordering vertices.

\begin{para}\label{par:dface}
	The \emph{degenerated faces} $\dface{i}\:\fn \to \dbound{\fn}{i}$ correspond to pairs of identities and simplicial degeneracies $(\id_{[n]},\sigma_i)$:
	\begin{equation*}
	\begin{tikzcd}
		{[m]}
			\ar[r, equal]
			\ar[d, "\fn"', twoheadrightarrow]
			\ar[dr,"\dface{i}" description, phantom] &
		{[m]}
			\ar[d, "\dbound{\fn}{i}", twoheadrightarrow] \\
		{[n]}
			\ar[r, "\sigma_i"', twoheadrightarrow] &
		{[n-1]}
	\end{tikzcd}
	\end{equation*}
	The codomain $\dbound{\fn}{i}=\ffib{\fn}{0} \osum\cdots\osum \ffib{\fn}{i} \mosum \ffib{\fn}{i+1} \osum\cdots\osum \ffib{\fn}{n}$ is obtained from $\fn$ by marking the unmarked edge between the fibres $\ffib{\fn}{i}$ and $\ffib{\fn}{i+1}$, thereby merging them.
	\begin{equation*}
	\begin{tikzcd}[column sep=0.01, row sep=0.4cm]
		0 
		\ar[ddr, mapsto]
		& \cdots
			\ar[dd, "{\scriptstyle\ffib{\fn}{0}}" description, phantom, near start]
		& \ep{}{0}
			\ar[ddl, mapsto]
		&& \ip{}{i}
			\ar[ddrr, mapsto]
		& \cdots
		& \ep{}{i} \ip{}{i+1}
			\ar[dd, "{\scriptstyle\ffib{\fn}{i} \mosum \ffib{\fn}{i+1}}" description, phantom, near start]
		& \cdots
		& \ep{}{i+1}
			\ar[ddll, mapsto]
		&& \ip{}{n}
			\ar[ddr, mapsto]
		& \cdots
			\ar[dd, "{\scriptstyle\ffib{\fn}{n}}" description, phantom, near start]
		& m
			\ar[ddl, mapsto] \\
		&&& \cdots &&&&&& \cdots & \\
			& 0 & &&&& i &&&&& n-1
	\end{tikzcd}
	\end{equation*}	
	More precisely, the epimorphism $\dbound{\fn}{i}\: [m] \xtwoheadrightarrow{} [n-1]$ maps an element $x$ to $\fn(x)$ if $x<\ep{}{i}$ and to $\fn(x)-1$ otherwise.
\end{para}
Degenerated faces are of particular interest as they encode the behaviour of degeneracies weakly.
Simplicial degeneracies usually send an edge to a vertex; here, degenerated faces modify objects of $\fDel$ by turning the $i^{\text{th}}$ unmarked edge into a marked edge.
They are the only type of faces which send unmarked edges to marked edges.

\begin{para}\label{par:vface}
	The \emph{vertical faces} $\vface{i}\: \vbound{\fn}{i} \to \fn$, for $i$ an inner marked vertex, are active maps of $\fDel$ composed of pairs of simplicial (active) faces and identities $(\face_i,\id_{[n]})$:
	\begin{equation*}
	\begin{tikzcd}
		{[m-1]}
			\ar[r, "\face_{i}", mid vert, hook]
			\ar[d, "\vbound{\fn}{i}"', twoheadrightarrow]
			\ar[dr,"\vface{i}" description, "\ulcorner" very near end, phantom] &
		{[m]}
			\ar[d, "\fn", twoheadrightarrow] \\
		{[n]}
			\ar[r, equal] &
		{[n]}
	\end{tikzcd}
	\end{equation*}
	The domain $\vbound{\fn}{i} = \ffib{\fn}{0} \osum\cdots\osum (\aint{i}{}-1) \osum\cdots\osum \ffib{\fn}{n}$ can be described as $\fn$ with the inner marked vertex $i$ removed.
	\begin{equation*}
	\begin{tikzcd}[column sep=0.01, row sep=0.4cm]
		0 
		\ar[ddr, mapsto]
		& \cdots
			\ar[dd, "{\scriptstyle\ffib{\fn}{0}}" description, phantom, near start]
		& \ep{}{0}
			\ar[ddl, mapsto]
		&& \ip{}{\fn(i)}
			\ar[ddrr, mapsto]
		& \cdots
		& \widehat{i}
			\ar[dd, "{\scriptstyle\aint{i}{}-1}" description, phantom, near start]
		& \cdots
		& \ep{}{\fn(i)}-1
			\ar[ddll, mapsto]
		&& \ip{}{n}-1
			\ar[ddr, mapsto]
		& \cdots
			\ar[dd, "{\scriptstyle\ffib{\fn}{n}}" description, phantom, near start]
		& m-1
			\ar[ddl, mapsto] \\
		&&& \cdots &&&&&& \cdots & \\
			& 0 & &&&& \fn(i) &&&&& n
	\end{tikzcd}
	\end{equation*}
	More precisely, the epimorphism $\vbound{\fn}{i}\:[m-1] \xtwoheadrightarrow{} [n]$ maps an element $x$ to $\fn(x)$ if ${x\leq\ep{}{\fn(i)} - 1}$ and $\fn(x-1)$ otherwise.
\end{para}
Vertical faces $\vface{i}$ modify objects of $\fDel$ by adding an inner marked vertex, thereby splitting a marked edge into two.
In particular, when restricted to the marked edges, they can be perceived as the usual simplicial face maps.

\begin{para}\label{par:sface}
	The \emph{standard faces} $\sface{i}\: \sbound{\fn}{i} \to \fn$, for $i$ a standard vertex, correspond to pairs of simplicial faces $(\face_{\bv{}{i}},\face_i)$:
	\begin{equation*}
	\begin{tikzcd}
		{[m-1]}
			\ar[r, "\face_{\bv{}{i}}",hook]
			\ar[d, "\sbound{\fn}{i}"', twoheadrightarrow]
			\ar[dr,"\sface{i}" description, phantom] &
		{[m]}
			\ar[d, "\fn", twoheadrightarrow] \\
		{[n-1]}
			\ar[r, "\face_i"',hook] &
		{[n]}
	\end{tikzcd}
	\end{equation*}
	The domain $\sbound{\fn}{i}=\ffib{\fn}{0} \osum\cdots\osum \widehat{\ffib{\fn}{i}} \osum\cdots\osum \ffib{\fn}{n}$ can be described as $\fn$ but with the standard vertex $\bv{}{i}$ removed.
	\begin{equation*}
	\begin{tikzcd}[column sep=0.01, row sep=0.4cm]
		0 
		\ar[ddr, mapsto]
		& \cdots
			\ar[dd, "{\scriptstyle\ffib{\fn}{0}}" description, phantom, near start]
		& \ep{}{0}
			\ar[ddl, mapsto]
		&& \widehat{\bv{}{i}}
			\ar[dd, "{\scriptstyle\ffib{\fn}{i}}" description, near start, mapsto]
		&& \ip{}{n}-1
			\ar[ddr, mapsto]
		& \cdots
			\ar[dd, "{\scriptstyle\ffib{\fn}{n}}" description, phantom, near start]
		& m-1
			\ar[ddl, mapsto] \\
		&&& \cdots && \cdots & \\
		  & 0 & && \widehat{i} &&& n-1
	\end{tikzcd}
	\end{equation*}
	More precisely, the epimorphism $\sbound{\fn}{i}\:[m-1] \xtwoheadrightarrow{} [n-1]$ maps an element $x$ to $\fn(x)$ if $x\leq\ep{}{i-1}$ and to $\fn(x)-1$ otherwise.
\end{para}
As the name suggests, the standard faces act as ordinary simplicial faces on the unmarked edges, with the additional property of preserving the marked ones.
Notably, the standard faces modify the objects of $\fDel$ by adding a fresh fibre $[0]^{\flat}$ (i.e.\ a standard vertex), thereby splitting an unmarked edge into two.

\begin{rmk}
	Observe that the squares for the standard faces $\sface{i}$ are always both pushouts and pullbacks in $\Delta$. In particular, the condition that $i$ is standard is redundant, but we keep it for clarity.
\end{rmk}

For vertical faces, the condition that the map is active in $\fDel$ ensures that the newly added vertex is inner marked.
In particular, the marked edge is inserted within the fibre and not attached to either of the extremities of the fibre.
The following type of morphisms encodes this operation.
However, it is essential to remark that they are not part of the generators, as they can be expressed as compositions of the degenerated and standard faces.
For these reasons, we have decided to separate them from the vertical faces.
 
\begin{para}\label{par:ovface}
	The \emph{bordering extensions} $\bbface{i}{\epsilon}\: \sbound{\fn}{i+\epsilon} \to \dbound{\fn}{i}$, for $\epsilon=0$ or $1$, correspond to pairs of simplicial faces and identities $(\face_{\bv{}{i+\epsilon}},\id_{[n-1]})$:
	\begin{equation*}
	\begin{tikzcd}[row sep=0.9cm]
		{[m-1]}
			\arrow[r, "\face_{\bv{}{i+\epsilon}}", hook]
			\arrow[d, "\sbound{\fn}{i+\epsilon}"', twoheadrightarrow] 
			\ar[dr,"\sface{i+\epsilon}" description, phantom]
		& {[m]}
			\arrow[r, equal]
			\arrow[d, "\fn" description, twoheadrightarrow]
			\ar[dr,"\dface{i}" description, phantom]
		& {[m]}
			\arrow[d, "\dbound{\fn}{i}", twoheadrightarrow] \\
		{[n-1]}
			\arrow[r, "\face_{i+\epsilon}"', hook]
			\ar[drru, equal, rounded corners, to path={|-([yshift=-0.3cm]\tikztostart.south)-| (\tikztotarget) \tikztonodes}]
		& {[n]}
			\arrow[r, "\sigma_i"', twoheadrightarrow]
		& {[n-1]}
	\end{tikzcd}
	\end{equation*}
	In addition, we say that $\bbface{i}{\epsilon}$ is a \emph{left} (resp.\ \emph{right}) bordering extension if $\epsilon=0$ (resp.\ $\epsilon=1$) and hence $\bv{}{i+\epsilon}$ is a left (resp.\ right) bordering vertex.
	Similarly to vertical faces, bordering extensions create a marked edge with the difference that it is inserted either at the beginning or at the end of the fibre $\ffib{\fn}{i+1-\epsilon}$.
	For instance, in the latter case (i.e.\ $\epsilon=1$) we have
	\belowdisplayskip=-5pt
	\begin{equation*}
	\begin{tikzcd}[column sep=0.01, row sep=0.4cm]
		0 
		\ar[ddr, mapsto]
		& \cdots
			\ar[dd, "{\scriptstyle\ffib{\fn}{0}}" description, phantom, near start]
		& \ep{}{0}
			\ar[ddl, mapsto]
		&& \ip{}{i}
			\ar[ddrr, mapsto]
		& \cdots
			\ar[ddr, "{\scriptstyle\ffib{\fn}{i}}" description, phantom, near start]
		& \ep{}{i}
			\ar[dd, mapsto]
		& \bv{}{i+1}
			\ar[ddl, "{\scriptstyle\ffib{\fn}{i+1}}" description, mapsto, near start]
		&
		& \ip{}{n}
			\ar[ddr, mapsto]
		& \cdots
			\ar[dd, "{\scriptstyle\ffib{\fn}{n}}" description, phantom, near start]
		& m
			\ar[ddl, mapsto] \\
		&&& \cdots &&&&& \cdots & \\
			& 0 & &&&& i &&&& n-1
	\end{tikzcd}
	\end{equation*}
\end{para}

Even though the bordering extensions are not part of the generators, they play a crucial r\^ole in the decomposition of \cref{thm:grdu}.
The bottom composition reduces to an identity, thereby preventing degenerated and standard faces from commuting properly, thus resulting in a special kind of map (see \cref{par:sd}).
In addition, these extensions belong to the class $\fVert$ of vertical morphisms, which plays as an important place in the study of weak identity structures via $\fDel$. In fact, they can be thought as encoding the source and target faces for (weak) degeneracies.
It is therefore useful to understand their interactions with the other faces.

\begin{lem}\label{lem:ofact}
	Any map of the form
	\begin{equation*}
	\begin{tikzcd}[row sep=1.2cm, column sep=1.2cm]
		{[m-1]}
			\ar[r, "\face_{i}", hook]
			\ar[d, "\fn\face_i" description, twoheadrightarrow]
		& {[m]}
			\ar[d, "\fn" description, twoheadrightarrow] \\
		{[n]}
			\ar[r, equal]
		& {[n]}
	\end{tikzcd}
	\end{equation*}
	which is not a vertical face is a bordering extension.
\end{lem}
\begin{proof}
	Take the pushout of the span to produce a standard face and use the universal property of pushouts to obtain the following commutative diagram:
	\begin{equation*}
	\begin{tikzcd}[row sep=1.2cm, column sep=1.2cm]
		{[m-1]}
			\arrow[r, "\face_{i}", hook]
			\arrow[d, "\fn\face_i" description, twoheadrightarrow] 
			\ar[dr,"\sface{\fk(i)}" description, "\ulcorner" very near end, phantom]
		& {[m]}
			\arrow[r, equal]
			\arrow[d, "\fk" description, twoheadrightarrow]
			\ar[dr, phantom]
		& {[m]}
			\arrow[d, "\fn" description, twoheadrightarrow] \\
		{[n]}
			\arrow[r, "\face_{\fk(i)}"', hook]
			\ar[drru, equal, rounded corners, to path={|-([yshift=-0.3cm]\tikztostart.south)-| (\tikztotarget) \tikztonodes}]
		& {[n+1]}
			\arrow[r, dotted, twoheadrightarrow]
		& {[n]}
	\end{tikzcd}
	\end{equation*}
	Since the bottom composite is the identity, the dotted arrow is a retraction of $\face_{\fk(i)}$.
	The latter has only two retractions: $\sigma_{\fk(i)}$ and $\sigma_{\fk(i)-1}$.
	Therefore, the right square is the degenerated face $\dface{\fk(i)-\epsilon}$, which give by definition the bordering extension $\bbface{\fk(i)-\epsilon}{\epsilon}$.
\end{proof}

\subsection{Fat simplicial relations}\label{sec:frel}
Since the degenerated, vertical and standard faces depend only on either the simplicial faces or the simplicial degeneracies, it is possible to derive a system of relations analogous to the simplicial identities of $\Delta$.

\begin{para}
	The self-interactions are straightforward and are given by the following relations:
	\belowdisplayskip=-5pt
	\begin{align}
		\dface{j}\dface{i} &= \dface{i}\dface{j+1} \hspace{0.725cm} i\leq j\leq n, \label{rel:dd}\\
		\sface{i}\sface{j} &= \sface{j+1}\sface{i} \hspace{0.725cm} i\leq j\leq n, \label{rel:hh}\\
		\vface{i}\vface{j} &= \vface{j+1}\vface{i} \hspace{0.725cm} i\leq j\leq m. \label{rel:vv}
	\end{align}
\end{para}

Most of the crossed-interactions are more intricate and require some clarification.
For each relation, we begin by explaining the interactions between the faces before giving the relations.

Recall from \cref{par:dface,par:vface,par:sface} that degenerated faces $\dface{i}$ turn unmarked edges into marked ones, thereby merging two fibres.
Standard faces $\sface{i}$ split unmarked edges into two by introducing standard vertices, while vertical faces $\vface{i}$ split marked edges into two by introducing inner marked vertices.
We start with the simplest case: the interaction between standard and vertical faces.

\begin{para}
    The two freshly produced standard and inner marked vertices do not interact with each other as they must be (at least) separated by a bordering vertex.
    In other words, since both top morphisms of $\sface{i}$ and $\vface{j}$ are simplicial faces, the associated simplicial relation applies as well, with the caveat that $\bv{}{i}$ and $j$ cannot be equal.
    In addition, the equality has to be split into two parts to obtain the correct indexation.
	Therefore, we have the following relation:
	\belowdisplayskip=-5pt
	\begin{equation}\label{rel:hv}
		\sface{i}\vface{j} =
		\begin{cases}
			\vface{j+1}\sface{i} &\bv{}{i} < j \leq m \\
			\vface{j}\sface{i}   &j < \bv{}{i}
		\end{cases}
	\end{equation}
\end{para}

\begin{rmk}
	In the second equality of \eqref{rel:hv}, observe that the top simplicial face of the standard face $\sface{i}$ is $\face_{\bv{}{i}-1}$.
\end{rmk}

The interactions between degenerated faces and the other faces are the most intricate.
We first explain the relation involving degenerated and vertical faces.

\begin{para}
    The two faces $\vface{j}$ and $\dface{i}$ can freely commute (thanks to the identity at the top and bottom) as long as the vertical face does not split the edge newly marked by the degenerated face.
    Indeed, if the edge is not marked, then the vertical face cannot insert an inner marked vertex to split it.
    Nonetheless, a relation can still be established for this case: instead of marking an edge and then splitting it, one can first insert a standard vertex with a standard face, which will act as the inner marked vertex, and then mark both of its neighbouring edges with degenerated faces.
    This results in the following relations:
	\begin{equation}\label{rel:dv}
		\vface{j}\dface{i} =
        \begin{cases}
            \dface{i}\bbface{i+1-\epsilon}{\epsilon} & j=\bv{}{i+1} \\
            \dface{i}\vface{j}   & j \neq \bv{}{i+1}
        \end{cases}
	\end{equation}
	for $\epsilon=0$ or $1$.
\end{para}

Finally, for the faces $\sface{i}$ and $\dface{j}$ to commute, one needs to ensure that the freshly marked edge does not interact with the standard vertex $\bv{}{i}$.
Otherwise, the commutativity would be obstructed as it is not possible to merge the standard vertex with another fibre if $\bv{}{i}$ has not been created yet.

\begin{para}\label{par:sd}
    The degenerated face $\dface{j}$ can merge a standard vertex, inserted by $\sface{i}$, to a fibre in two ways: either by the start point or by the end point of a fibre.
	The bordering extensions naturally emerge from this obstruction.
	In particular, as outlined in \cref{par:ovface}, the two cases to avoid for the relation to hold are for $j=i$ and for $j=i-1$.
    Meaning that, using the simplicial relation for simplicial faces and degeneracies, we can reduce the bottom composition to an identity, thereby obtaining a bordering extension.
	At the end, we derive the following relation:
	\begin{align}\label{rel:hd}
		\dface{j}\sface{i} &=
		\begin{cases}
			\sface{i}\dface{j-1} & i<j\leq n \\
			\bbface{j}{\epsilon}	   & i=j+\epsilon \\
			\sface{i-1}\dface{j} & j+1<i \leq n
		\end{cases}
	\end{align}
	for $\epsilon=0$ or $1$.
\end{para}

\begin{rmk}
	We emphasise that the second equality holds by definition \cref{par:ovface}.
	Indeed, this a very special case of the composition of standard and degenerated faces which happens to reduce to a morphism very similar to a vertical face.
\end{rmk}

It is very convenient to consider bordering extensions as a distinct "face", as the obstruction mentioned above occurs naturally and cannot be avoided, as depicted in relation \eqref{rel:dv}.
For this reason, the following collection of relations describing the interactions with the bordering extensions will be convenient.

\begin{prop}\label{prop:ovrel}
    We can derive the following relations about the bordering extensions:
	\vspace{5pt}
	\noindent
	\begin{minipage}{0.42\linewidth}
	\begin{flalign}
		&\bbface{j}{\epsilon}\bbface{i}{1-\epsilon} = \bbface{i}{1-\epsilon}\bbface{j}{\epsilon} & \label{rel:ww1} \\
		&\bbface{j}{\epsilon}\sface{i} =
		\begin{cases}
			\sface{i}\bbface{j}{\epsilon} & i > j \\
			\bbface{j}{1-\epsilon}\sface{i} & i=j \\
			\sface{i}\bbface{j-1}{\epsilon} & i < j
		\end{cases} & \label{rel:wd}
	\end{flalign}
	\end{minipage}
	\hfill
	\noindent
	\begin{minipage}{0.54\linewidth}
	\begin{flalign}
		&\bbface{j}{\epsilon}\bbface{i}{\epsilon} = 
		\begin{cases}
			\bbface{i}{\epsilon}\bbface{j}{\epsilon} & i \neq j \\
			\vface{\bv{}{i+\epsilon}+1-\epsilon}\bbface{i}{\epsilon} & i = j
		\end{cases} & \label{rel:ww2}\\
		&\vface{j}\bbface{i}{\epsilon} =
		\begin{cases}
			\bbface{i}{\epsilon}\vface{j} & j < \bv{}{i+\epsilon}+1-\epsilon \\
			\bbface{i}{\epsilon}\bbface{i}{\epsilon} & j = \bv{}{i+\epsilon}+1-\epsilon \\
			\bbface{i}{\epsilon}\vface{j-1} & j > \bv{}{i+\epsilon}+1-\epsilon
		\end{cases} & \label{rel:vw}
	\end{flalign}
	\end{minipage}
	\vspace{5pt}
	\noindent
	\begin{minipage}{\linewidth}
	\begin{equation}\label{rel:sw}
		\dface{j}\bbface{i}{\epsilon}=
		\begin{cases}
			\bbface{i}{\epsilon}\dface{j} & j > i-1+\epsilon \\
			\vface{\bv{}{i+\epsilon}}\dface{j} = \dface{j}\bbface{i+(-1)^{1-\epsilon}}{1-\epsilon} & j=i-1+\epsilon \\
			\bbface{i-1}{\epsilon}\dface{j} & j < i-1+\epsilon
		\end{cases}
	\end{equation}		
	\end{minipage}
\end{prop}
\begin{proof}
	Follows from the relations \eqref{rel:dd}--\eqref{rel:hd}.
	We treat the case \eqref{rel:ww1}, the other relations are shown similarly.
	Since bordering extensions are composed of a standard face followed by a degenerated face we have, for $\epsilon=0$ or $1$,
	\begin{equation*}
		\bbface{j}{\epsilon}\bbface{i}{1-\epsilon} =\dface{j}\sface{j+\epsilon}\dface{i}\sface{i+1-\epsilon}.
	\end{equation*}
	Therefore, using the first and second equalities of \eqref{rel:hd} we obtain
	\begin{equation*}
		\bbface{j}{\epsilon}\bbface{i}{1-\epsilon} =\dface{j}\sface{j+\epsilon}\dface{i}\sface{i+1-\epsilon}=
		\begin{cases}
			\dface{j}\dface{i+1}\sface{j+\epsilon}\sface{i+1-\epsilon} & j+\epsilon < i+1 \\
			\dface{j}\dface{i}\sface{j+1+\epsilon}\sface{i+1-\epsilon} & j+\epsilon > i.
		\end{cases}
	\end{equation*}
	To simplify the tracking of the indices, we evaluate $\epsilon$ at $0$ first and then $1$.
	Thus, for $\epsilon=0$ and $j< i+1$ (on the left) and $j>i$ (on the right), we have
	\smallskip

	\noindent
	\begin{minipage}{0.46\linewidth}
	\begin{flalign*}
		\dface{j}\dface{i+1}\sface{j}\sface{i+1} &= \dface{j}\dface{i+1}\sface{i+2}\sface{j} &\text{by \eqref{rel:hh}} \\
		&=\dface{i}\dface{j}\sface{i+2}\sface{j} &\text{by \eqref{rel:dd}} \\
		&=\dface{i}\sface{i+1}\dface{j}\sface{j} &\text{by \eqref{rel:hd}}\\
		&=\bbface{i}{1}\bbface{j}{0}
	\end{flalign*}		
	\end{minipage}
	\hfill
	\noindent
	\begin{minipage}{0.46\linewidth}
	\begin{flalign*}
		\dface{j}\dface{i}\sface{j+1}\sface{i+1} &= \dface{j}\dface{i}\sface{i+1}\sface{j} &\text{by \eqref{rel:hh}} \\
		&=\dface{i}\dface{j+1}\sface{i+1}\sface{j} &\text{by \eqref{rel:dd}} \\
		&=\dface{i}\sface{i+1}\dface{j}\sface{j} &\text{by \eqref{rel:hd}}\\
		&=\bbface{i}{1}\bbface{j}{0}
	\end{flalign*}
	\end{minipage}
	\bigskip

	\noindent
	Then, for $\epsilon=1$ and $j<i$ (on the left) and $j+1>i$ on the right, we have
	\smallskip

	\noindent
	\begin{minipage}{0.46\linewidth}
	\begin{flalign*}
		\dface{j}\dface{i+1}\sface{j+1}\sface{i} &= \dface{j}\dface{i+1}\sface{i+1}\sface{j+1} &\text{by \eqref{rel:hh}} \\
		&=\dface{i}\dface{j}\sface{i+1}\sface{j+1} &\text{by \eqref{rel:dd}} \\
		&=\dface{i}\sface{i}\dface{j}\sface{j+1} &\text{by \eqref{rel:hd}}\\
		&=\bbface{i}{0}\bbface{j}{1}
	\end{flalign*}		
	\end{minipage}
	\hfill
	\bigskip
	\noindent
	\begin{minipage}{0.46\linewidth}
	\begin{flalign*}
		\dface{j}\dface{i}\sface{j+2}\sface{i} &= \dface{j}\dface{i}\sface{i}\sface{j+1} &\text{by \eqref{rel:hh}} \\
		&=\dface{i}\dface{j+1}\sface{i}\sface{j+1} &\text{by \eqref{rel:dd}} \\
		&=\dface{i}\sface{i}\dface{j}\sface{j+1} &\text{by \eqref{rel:hd}}\\
		&=\bbface{i}{0}\bbface{j}{1}
	\end{flalign*}
	\end{minipage}
	The other relations are shown by similar calculation and careful treatment of the indexation.
\end{proof}

\subsection{Factorisation into faces}\label{sec:fact}
In this section, we prove that any morphism in $\fDel$ can be factored as a composition of degenerated, standard and vertical faces (\cref{thm:genreldec}).
The key steps in the proof consist of factoring the bottom arrow of a given morphism $\fn \to \fk$ in $\fDel$ into an epimorphism followed by a monomorphism, and then decomposing the resulting maps into compositions of degenerated, standard and vertical faces.
The epimorphism case is relatively straightforward as it only involves degenerated faces, whereas the monomorphism case is more complex.
To deal with this, we split the latter case and start with a simpler subcase.

\begin{lem}\label{lem:vfact}
	Any vertical morphism $f\:\fn \to \fk$ in $\fDel$, i.e.\ such that the bottom arrow $\fbot{f}$ is an identity as in \eqref{para:tfact}, can be factored as a composition of standard, vertical and degenerated faces as follows:
	\begin{equation}\label{diag:vfact}
	\begin{tikzcd}[row sep=1.2cm, column sep=1.2cm]
		\cdot
			\ar[r, "\ftop{\phi}", hook]
			\ar[d, "\fn" description, twoheadrightarrow]
			\ar[dr, "\phi" description, phantom]
		& \cdot
			\ar[r, "\ftop{\nu}", hook]
			\ar[d, "\fk\ftop{\nu}" description, twoheadrightarrow]
			\ar[dr, "\nu" description, phantom] 
		& \cdot
			\ar[d, "\fk" description, twoheadrightarrow] \\
		\cdot
			\ar[r, equal]
		& \cdot
			\ar[r, equal]
		& \cdot
	\end{tikzcd}
	\end{equation}
	with $\nu=\vface{k_{\gamma}}\cdots\vface{k_{1}}$ and $\phi=\bbface{c_{\theta}}{\epsilon_{\theta}}\cdots\bbface{c_1}{\epsilon_1}$, and where the indices are in increasing order.
\end{lem}
\begin{proof}
    The bottom map is an identity, consequently $f$ does not create new standard vertices and only creates inner marked and bordering vertices.
    More precisely, $f$ only interacts with the fibres $\ffib{\fn}{p}$, for $p$ in the codomain of $\fn$, by attaching bordering vertices at the extremities and incrementally inserting inner marked vertices until the desired length $\ffib{\fk}{p}$ is reached.
    One can visualise the action of $f$ restricted to the fibre $\ffib{\fn}{p}$ as follows.
	\begin{equation*}
	\begin{tikzcd}[column sep=0.01,row sep=small]
		q & \cdots & {q'} &{}&{}&{}&{}& \ip{}{p} & \cdots & {f(q)} & \cdots & {f(q')} & \cdots & \ep{}{p} \\
		\cdots && \cdots &&&&& \cdots &&&&&& \cdots \\
		& p &&&&&&&&& p
		\arrow[maps to, from=1-1, to=3-2]
        \arrow[from=1-2, to=3-2, "\ffib{\fn}{p}" description, near start, draw=none]
		\arrow[maps to, from=1-3, to=3-2]
		\arrow[maps to, from=1-8, to=3-11]
		\arrow[from=1-11, to=3-11, "\ffib{\fk}{p}" description, near start, draw=none]
		\arrow[maps to, from=1-10, to=3-11]
		\arrow[maps to, from=1-12, to=3-11]
		\arrow[maps to, from=1-14, to=3-11]
		\arrow["f", from=2-3, to=2-8]
	\end{tikzcd}
	\end{equation*}
    Since the top map $\ftop{f}$ is a monomorphism in $\Delta$ it can be canonically factored as a composition of simplicial faces.
    Therefore, we can factor $\ftop{f}$ by first putting the simplicial faces $\face_{\bv{\epsilon_{\theta}}{c_{\theta}}}\cdots\face_{\bv{\epsilon_{1}}{c_1}}=\ftop{\phi}$ attaching all the bordering vertices that are not in the image of $\ftop{f}$ (such as $\face_{\ip{}{p}}$ and/or $\face_{\ep{}{p}}$), followed by the rest of the faces  $\face_{k_{\gamma}}\cdots\face_{k_1}=\ftop{\nu}$ inserting the inner marked vertices, both in increasing order.
    In particular, we obtain a factorisation of $f$ as in \eqref{diag:vfact}:
	\begin{equation*}
		f=\phi\nu.
	\end{equation*}
	It remains to show that $\phi$ and $\nu$ factor into degenerated, standard and vertical faces.
	For the morphism $\phi$, since each simplicial face $\face_{\bv{\epsilon_p}{c_p}}$, with $1\leq p\leq \theta$, creates a bordering vertex, we can repeatedly apply the construction of \cref{lem:ofact} and obtain a composition of bordering extensions $\phi = \bbface{c_{\theta}}{\epsilon_{\theta}}\cdots\bbface{c_1}{\epsilon_1}$.
    \begin{equation*}
	\begin{tikzcd}[row sep= 1.2cm, column sep=1.5cm]
		\cdots
			\ar[r, hook]
			\ar[dr, "\bbface{c_{p+1}}{\epsilon_{p+1}}" description, phantom]
		& \cdot
			\ar[r, "\face_{\bv{\epsilon_p}{c_p}}", hook]
			\ar[d, twoheadrightarrow]
			\ar[dr, "\sface{c_p+\epsilon_p}" description, "\ulcorner" very near end, phantom]
		& \cdot
			\ar[r, equal]
			\ar[d, twoheadrightarrow]
			\ar[dr, "\dface{c_p}" description, phantom]
		& \cdot
			\ar[d, twoheadrightarrow]
            \ar[r, hook]
		& \cdots \\
		\cdots
			\ar[r, equal]
		& \cdot
			\ar[r, "\face_{c_p+\epsilon_p}"', hook]
			\ar[drru, equal, rounded corners, to path={|-([yshift=-0.3cm]\tikztostart.south)-| (\tikztotarget) \tikztonodes}]
		& \cdot
			\ar[r, "\sigma_{c_p}"', twoheadrightarrow]
		& \cdot
			\ar[r, equal]
		& \cdots
	\end{tikzcd}
	\end{equation*}
	For the second composition, since each simplicial face $\face_{k_p}$, with $1\leq p\leq \gamma$, creates an inner marked vertex within the fibre $\ffib{\fn}{\fk(k_p)}$ to ultimately obtain $\aint{k_p}{\fk}$, they are active morphisms of $\Delta$.
	Taking the pushout along $\face_{k_p}$ yields a vertical face.
	\begin{equation*}
	\begin{tikzcd}[row sep= 1.2cm, column sep=1.2cm]
		\cdots
			\ar[r, hook]
		& \cdot
			\ar[r, "\face_{k_p}", mid vert, hook]
			\ar[d, twoheadrightarrow]
			\ar[dr, phantom, "\vface{k_p}", "\ulcorner" very near end]
		& \cdot
			\ar[d, twoheadrightarrow]
			\ar[r, hook]
		& \cdots \\
		\cdots
			\ar[r, equal]
		& \cdot
			\ar[r, equal]
		& \cdot
			\ar[r, equal]
		& \cdots
	\end{tikzcd}
	\end{equation*}
	Consequently, we obtain a composition of vertical faces $\nu=\vface{k_{\gamma}}\cdots\vface{k_1}$.
\end{proof}

As morphisms in $\fDel$ are represented by commutative squares as in \eqref{diag:mofDel}, when the bottom map is also a monomorphism, the top monomorphism must factor into at least as many faces as those decomposing the bottom arrow.
The surplus of faces factoring the top map can then be rearranged to form a composition of maps with an identity at the bottom.
This where \cref{lem:vfact} becomes useful.
In the following lemma we make this point precise for the case of a monomorphism at the bottom.

\begin{rmk}\label{lem:pbpe}
    In $\Delta$, pullbacks along monomorphisms preserve epimorphisms. Indeed, if
	\begin{equation*}
    \begin{tikzcd}
        {[m]}
            \arrow[r, "\ftop{f}", hook]
            \arrow[d, "\fn"']
            \arrow[dr, phantom, very near start, "{ \lrcorner }"]
        & {[m']}
            \arrow[d, "\fk", twoheadrightarrow] \\
        {[n]}
            \arrow[r, "\fbot{f}"', hook]
        & {[n']}
    \end{tikzcd}
    \end{equation*}
	is a pullback square in $\Delta$, observe that $\fn$ can be decomposed as the fibres of $\fk$ over the image of $\fbot{f}$: $\fn= \ffib{\fk}{\fbot{f}(0)} \osum\cdots\osum \ffib{\fk}{\fbot{f}(n)}$, and hence is an epimorphism.
\end{rmk}

\begin{lem}\label{lem:monofact}
	Any horizontal morphism $f\:\fn \to \fk$ in $\fDel$, i.e.\ such that the bottom arrow $\fbot{f}$ is a monomorphism as in \eqref{para:tfact}, can be factored as a composition of degenerated, standard and vertical faces as follows:
	\begin{equation}\label{diag:monofact}
	\begin{tikzcd}[row sep=1.2cm, column sep=1.2cm]
		\cdot
			\ar[d, "\fn" description, twoheadrightarrow]
			\ar[r, "\ftop{\phi}", hook]
			\ar[dr, "\phi", phantom]		
		& \cdot
			\ar[r, "\ftop{\nu}", hook]
			\ar[d, "\fm\ftop{\nu}" description, twoheadrightarrow]
			\ar[dr, "\nu", phantom]		
		& \cdot
			\ar[r, "\ftop{\delta}", hook]
			\ar[d, "\fm" description, twoheadrightarrow]
			\ar[dr, phantom, "\face"]
		& \cdot
			\ar[d, "\fk\ftop{\tau}\ftop{\psi}" description, twoheadrightarrow]
			\ar[r, "\ftop{\psi}", hook]
			\ar[dr, "\psi", phantom]
		& \cdot
			\ar[d, "\fk\ftop{\tau}" description, twoheadrightarrow]
            \ar[r, "\ftop{\tau}", hook]
            \ar[dr, "\tau", phantom]
        & \cdot
        \ar[d, "\fk" description, twoheadrightarrow] \\
		\cdot
			\ar[r, equal]
		& \cdot
			\ar[r, equal]		
		& \cdot
			\ar[r, "\fbot{\delta}"', hook]
		& \cdot
			\ar[r, equal]
		& \cdot
            \ar[r, equal]
        & \cdot
	\end{tikzcd}
	\end{equation}
	with the decompositions ${\tau=\vface{t_{\iota}}\cdots\vface{t_1}}$, ${\psi=\bbface{i_{\alpha}}{\epsilon_{\alpha}}\cdots\bbface{i_{1}}{\epsilon_1}}$, ${\delta=\sface{j_{\beta}}\cdots\sface{j_{1}}}$, ${\nu=\vface{k_{\gamma}}\cdots\vface{k_{1}}}$ and ${\phi=\bbface{c_{\theta}}{\epsilon_{\theta}}\cdots\bbface{c_1}{\epsilon_1}}$, and where the indices are in increasing order.
\end{lem}
\begin{proof}
	Let $f\:\fn \to \fk$ be as above.
	To begin, take the pullback of the corresponding cospan and construct the following diagram.
	\begin{equation}\label{diag:pbmono}
	\begin{tikzcd}[row sep=1.2cm, column sep=1.2cm]
		\cdot
			\ar[r, hook, dotted]
			\ar[d, "\fn" description, twoheadrightarrow]
			\ar[urrd, "\ftop{f}" description, hook, rounded corners,
				to path={[pos=0.25]
					-- ([yshift=0.3cm] \tikztostart.north)
					-| (\tikztotarget) \tikztonodes}]
			\ar[dr, "g" description, phantom]
		& \cdot
			\ar[r, hook]
			\ar[d, "\fm" description, twoheadrightarrow]
			\ar[dr, phantom, "h", "{ \lrcorner }" very near start]
		& \cdot
			\ar[d, "\fk", twoheadrightarrow] \\
		\cdot
			\ar[r, equal]
		& \cdot
			\ar[r, "\fbot{f}"', hook]
		& \cdot
	\end{tikzcd}
	\end{equation}
	The arrow $\fm$ is an epimorphism by \cref{lem:pbpe} and, since the dotted arrow is a monomorphism in $\Delta$ (as $\ftop{f}$ is), the left square of \eqref{diag:pbmono} represents a morphism $g\:\fn \to \fm$ in~$\fDel$.
    Therefore, we obtain the factorisation $f=hg$ in $\fDel$.
	Since the bottom map $\fbot{g}$ is an identity, we apply \cref{lem:vfact} to obtain the factorisation \eqref{diag:vfact}:
	\begin{equation*}
		g=\nu\phi=\vface{k_{\gamma}}\cdots\vface{k_{1}}\bbface{c_{\theta}}{\epsilon_{\theta}}\cdots\bbface{c_1}{\epsilon_1}.
	\end{equation*}
	Next, factor $h$ as a composition of a horizontal morphism $\face$ followed by a vertical morphism $\omega$, as described in \cite[Lemma 6.11]{Pao2024},
	\begin{equation*}
	\begin{tikzcd}[row sep=1.2cm, column sep=1.2cm]
		\cdot
			\ar[r, "\ftop{\face}", hook]
			\ar[d, "\fm" description, twoheadrightarrow]
			\ar[dr, "\face", phantom]
		& \cdot
			\ar[r, "\ftop{\omega}", hook]
			\ar[d, "\fk\ftop{\omega}" description, twoheadrightarrow]
			\ar[dr, "\omega", phantom]
		& \cdot
			\ar[d, "\fk" description] \\
		\cdot
			\ar[r, "\fbot{f}"', hook]
		& \cdot
			\ar[r, equal]
		& \cdot
	\end{tikzcd}
	\end{equation*}
    More precisely, using the decomposition of $\fm$ from \cref{lem:pbpe}, it is not difficult to see that $h\: \fm \to \fk$ constructs all the fibres that are not pulled back, i.e.\ that are not in the image of $\fbot{f}$.
    We can split this construction in two steps: first the new fibres are introduced by inserting all standard vertices with $\delta$, then they are filled by $\omega$ to obtain the desired length.
	Since $\ftop{\face}$ is a monomorphism of $\Delta$, it can be canonically factored it as a composition $\face_{\bv{}{j_{\beta}}}\cdots\face_{\bv{}{j_{1}}}$ of simplicial faces in increasing order.
	These faces insert standard vertices which correspond to the first vertices of the fibres that are not in the image of $\fbot{f}$.
	Take the iterated pushouts along the faces to build a composition of standard faces:
	\begin{equation*}
	\begin{tikzcd}[row sep=1.2cm, column sep=1.8cm]
		\cdot
			\ar[r, "\face_{\bv{}{j_1}}", hook]
			\ar[d, "\fm" description, twoheadrightarrow]
			\ar[dr, "\sface{j_1}", "\ulcorner" very near end,  phantom]
		& \cdot
			\ar[r, "\cdots" description, phantom]
			\ar[d, "\sbound{\fk\ftop{\omega}}{j_{\beta}\cdots j_2}" description, twoheadrightarrow]
			\ar[dr, "\ulcorner" description, very near end, phantom]
		& \cdot
			\ar[r, "\face_{\bv{}{j_{\beta}}}", hook]
			\ar[d, "\sbound{\fk\ftop{\omega}}{j_1}" description, twoheadrightarrow]
			\ar[dr, "\sface{j_{\beta}}", "\ulcorner" very near end,  phantom]
		& \cdot
			\ar[d, "\fk\ftop{\omega}" description] \\
		\cdot
			\ar[r, "\face_{j_1}"', hook]
		& \cdot
			\ar[r, "\cdots" description, phantom]
		& \cdot
			\ar[r, "\face_{j_{\beta}}"', hook]
		& \cdot
	\end{tikzcd}
	\end{equation*}
	By the pushout pasting lemma, the outer square is a pushout, and hence the bottom composition is the factorisation of $\fbot{f}$ into simplicial faces.
    We obtain the factorisation $\delta=\sface{j_{\beta}}\cdots\sface{j_{1}}$, where the indices are in increasing order.
    Finally, we use \cref{lem:monofact} to factor $\omega$ into $\psi\tau$, where $\psi=\bbface{i_{\alpha}}{\epsilon_{\alpha}}\cdots\bbface{i_1}{\epsilon_1}$ and $\tau=\vface{t_{\iota}}\cdots\vface{t_1}$.
\end{proof}

We are now ready to construct the full factorisation of an arbitrary morphism in $\fDel$.

\begin{thm}\label{thm:genreldec}
    Any map $f\:\fn \to \fk$ in $\fDel$ as in \eqref{diag:mofDel} factors as a composition of degenerated, vertical and standard faces, as follows:
    \begin{equation}\label{diag:genreldec}
	\begin{tikzcd}[row sep=1.2cm, column sep=1.2cm]
        \cdot
            \ar[r, equal]
            \ar[d, "\fn"' description, twoheadrightarrow]
            \ar[dr, "\sigma", phantom]
		& \cdot
			\ar[d, "\fbot{\sigma}\fn" description, twoheadrightarrow]
			\ar[r, "\ftop{\phi}", hook]
			\ar[dr, "\phi", phantom]		
		& \cdot
			\ar[r, "\ftop{\nu}", hook]
			\ar[d, "\fm\ftop{\nu}" description, twoheadrightarrow]
			\ar[dr, "\nu", phantom]		
		& \cdot
			\ar[r, "\ftop{\delta}", hook]
			\ar[d, "\fm" description, twoheadrightarrow]
			\ar[dr, phantom, "\face"]
		& \cdot
			\ar[d, "\fk\ftop{\tau}\ftop{\psi}" description, twoheadrightarrow]
			\ar[r, "\ftop{\psi}", hook]
			\ar[dr, "\psi", phantom]
		& \cdot
			\ar[d, "\fk\ftop{\tau}" description, twoheadrightarrow]
            \ar[r, "\ftop{\tau}", hook]
            \ar[dr, "\tau", phantom]
        & \cdot
        \ar[d, "\fk" description, twoheadrightarrow] \\
        \cdot
            \ar[r, "\fbot{\sigma}"', twoheadrightarrow]
		& \cdot
			\ar[r, equal]
		& \cdot
			\ar[r, equal]		
		& \cdot
			\ar[r, "\fbot{\delta}"', hook]
		& \cdot
			\ar[r, equal]
		& \cdot
            \ar[r, equal]
        & \cdot
	\end{tikzcd}
	\end{equation}
    with the decompositions ${\tau=\vface{t_{\iota}}\cdots\vface{t_1}}$, ${\psi=\bbface{i_{\alpha}}{\epsilon_{\alpha}}\cdots\bbface{i_{1}}{\epsilon_1}}$, ${\delta=\sface{j_{\beta}}\cdots\sface{j_{1}}}$, ${\nu=\vface{k_{\gamma}}\cdots\vface{k_{1}}}$, ${\phi=\bbface{c_{\theta}}{\epsilon_{\theta}}\cdots\bbface{c_1}{\epsilon_1}}$ and $\sigma=\dface{l_{\lambda}}\cdots\dface{l_{1}}$, and where the indices are in increasing order.
\end{thm}
\begin{proof}
    First, factor $\fbot{f}$ uniquely as an epimorphism $\fbot{\sigma}$ followed by a monomorphism $\fbot{\face}$ in $\Delta$ to obtain
    \begin{equation*}
    \begin{tikzcd}[row sep=1.2cm, column sep=1.2cm]
        \cdot
            \ar[r, equal]
            \ar[d, "\fn" description, twoheadrightarrow]
            \ar[dr, "\sigma" description, phantom]
        & \cdot
            \ar[r, "\ftop{f}", hook]
            \ar[d, "\fbot{\sigma}\fn" description, twoheadrightarrow]
        & \cdot
            \ar[d, "\fk" description, twoheadrightarrow] \\
		\cdot
            \ar[r, "\fbot{\sigma}"', twoheadrightarrow]
        & \cdot
            \ar[r, "\fbot{\delta}"', hook]
        & \cdot
    \end{tikzcd}
    \end{equation*}
    Then apply \cref{lem:monofact} to decompose the right square into standard, degenerated and vertical faces to obtain the factorisation as in \eqref{diag:genreldec} in $\fDel$.
    The epimorphism $\fbot{\sigma}$ can be canonically factored as a composition of simplicial degeneracies $\sigma_{l_{\lambda}}\cdots\sigma_{l_{1}}$ in increasing order. Using that, we can factor $\sigma$ as a composition of degenerated faces $\dface{l_{\lambda}}\cdots\dface{l_{1}}$ without difficulties:
    \begin{equation*}
    \begin{tikzcd}[row sep=1.2cm, column sep=1.8cm]
        \cdot 
            \ar[r, equal]
            \ar[d, "\fn" description, twoheadrightarrow]
            \ar[dr, "\dface{l_1}",  phantom]
        & \cdot
            \ar[r, "\cdots" description, phantom]
            \ar[d, "\dbound{\fn}{l_1}" description, twoheadrightarrow]
        & \cdot
            \ar[r, equal]
            \ar[d, "\dbound{\fn}{l_{\lambda-1}\cdots l_1}" description, twoheadrightarrow]
            \ar[dr, "\dface{l_{\lambda}}",  phantom]
        & \cdot
            \ar[d, "\fbot{\sigma}\fn" description] \\
        \cdot
            \ar[r, "\sigma_{l_1}"', twoheadrightarrow]
        & \cdot
            \ar[r, "\cdots" description, phantom]
        & \cdot
            \ar[r, "\sigma_{l_{\lambda}}"', twoheadrightarrow]
        & \cdot
    \end{tikzcd}
    \end{equation*}
    In conclusion, we have the following factorisation
	\belowdisplayskip=-5pt
    \begin{equation}\label{eq:dvhdec}
    f=\vface{t_{\iota}}\cdots\vface{t_1}\bbface{i_{\alpha}}{\epsilon_{\alpha}}\cdots\bbface{i_{1}}{\epsilon_1}\sface{j_{\beta}}\cdots\sface{j_{1}}\vface{k_{\gamma}}\cdots\vface{k_{1}}\bbface{c_{\theta}}{\epsilon_{\theta}}\cdots\bbface{c_1}{\epsilon_1}\dface{l_{\lambda}}\cdots\dface{l_{1}}	
    \end{equation}
\end{proof}

\subsection{Uniqueness of the factorisation}\label{sec:uniq}
We now prove that the factorisation of \cref{thm:genreldec} is unique up to the relations \eqref{rel:dd}--\eqref{rel:hd}.
Our strategy to show the uniqueness is always the same: it is proven by induction on the number of degenerated, standard and vertical faces in the factorisations.

\begin{para}
	In the rest of this section, we will use the variable $p$ to represent the index of the faces which need to be relocated in the factorisations.
	When commuting the faces, some of the relations \eqref{rel:dd}--\eqref{rel:sw} change the indices.
	As a result, the index $p$ might change as the face is moved through the compositions of other faces appearing in the factorisation of \cref{lem:monofact,lem:vfact,thm:genreldec}.
	For example, consider the composition $\dface{p}\dface{k_n}\cdots\dface{k_1}$, where the indices $k_1$, \ldots, $k_n$ are in increasing order.
	There exists an $n_0$ such that $k_{n_0+1} \geq p \geq k_{n_0}$.
	Using the relation \eqref{rel:dd}, we can commute $\dface{p}$ through the composition of degenerated faces to obtain $\dface{k_n+1}\cdots\dface{k_{n_0}+1}\dface{p}\dface{k_{n_0}}\cdots\dface{k_1}$, and then commute it further to get $\dface{k_n+1}\cdots\dface{k_{n_0}+1}\dface{k_{n_0}}\cdots\dface{k_1}\dface{p-n_0}$.
	When handling similar cases repeatedly, the new indices can become cumbersome to track and carry around.
	Furthermore, depending on the situation, different cases may yield different new indices.
	To simplify notation, we will denote the new index as $\dvp$, treating $p$ and $\dvp$ as dynamic variables.
	This approach allows us to focus on the proof's structure without being encumbered by the exact expressions of the indices.
\end{para}

In order to break the proof in simpler subcases, we first show that each factorisation of \cref{lem:vfact,lem:monofact} for vertical and horizontal morphisms is unique up to the relations.

\begin{lem}\label{lem:reloc}
	Suppose a map $f\:\fn \to \fk$ in $\fDel$ decomposes into a single kind of face, i.e.\ $f=f_{k_{n}}\cdots f_{k_1}$ where all $f_p$'s are either $\dface{p}$, $\sface{p}$, $\vface{p}$ or $\bbface{p}{\epsilon_p}$.
	Then, the factorisation can be rearranged into a composition such that the indices are in increasing order.
\end{lem}
\begin{proof}
	We proceed by induction on the number $n$ of faces $f_{k_n}, \ldots, f_{k_1}$ in the decomposition of $f$. The argument is the same for all cases, so we present it in general.
	The base case $n=0$ is trivial.
	Suppose any factorisation of length $n$ and into a single kind of face can be rearranged in an increasing order and assume $f$ can be decomposed into $n+1$ faces as follows:
	\begin{equation*}
		f=f_pf_{k_{n}}\cdots f_{k_1}.
	\end{equation*}
	By hypothesis, the composition $f_{k_n}\cdots f_{k_1}$ can be rearranged into a composition of the form $f_{k'_{n}}\cdots f_{k'_1}$ where the indices $k'_1$, \ldots, $k'_n$ are in increasing order.
	If $p \geq k'_n$, then we are done.
	Otherwise, there is a $1 \leq n_0 < n$ such that $k'_{n_0+1} \geq \dvp \geq k'_{n_0}$.
	We use the appropriate relation (\eqref{rel:dd} for degenerated faces, \eqref{rel:hh} for standard faces, \eqref{rel:vv} for vertical faces, or \eqref{rel:ww1} and \eqref{rel:ww2} for bordering extensions) to move $f_p$ to its correct position in the factorisation, ensuring the indices are in increasing order.
\end{proof}

\begin{lem}\label{lem:vfactuni}
    The factorisation of \cref{lem:vfact} is unique up to the relations \eqref{rel:dd}--\eqref{rel:hd}.
\end{lem}
\begin{proof}
	By induction on the number $c$ of degenerated, standard and vertical faces in the factorisation of \cref{lem:vfact}.
	The case $c=0$ is trivial.
    Suppose any factorisation of length $c$ is unique up to the relations, and assume that $f$ decomposes into $c+1$ degenerated, standard and vertical faces.
	In particular, we have that $f=gf'$ or $f=f'g$, where $g$ is either a degenerated, standard or vertical face.
    Consequently, the map $f'$ has a factorisation of length $c$.
	Suppose $f=\vface{p}f'$, i.e.\ $g$ is a vertical face $\vface{p}$.
	By hypothesis, the factorisation of $f'$ is unique (up to the relations) and thus, can take the form of \cref{lem:vfact}:
    \begin{equation*}
        f'=\nu'\phi'=\vface{k_{\gamma}}\cdots\vface{k_{1}}\bbface{c_{\theta}}{\epsilon_{\theta}}\cdots\bbface{c_1}{\epsilon_1},
    \end{equation*}
	where the faces are arranged in increasing order.
	We conclude by  \cref{lem:reloc}.

	Suppose $f=\dface{p}f'$, i.e.\ $g$ is a degenerated face $\dface{p}$.
	The bottom composition $\sigma_p\fbot{f}'$ is the identity, hence $\fbot{f}'$ is a section of $\sigma_p$.
	The latter has only two sections: $\face_{p}$ or $\face_{p+1}$.
	In either case, we factor $f'$ as in \cref{lem:monofact} and obtain
	\begin{equation*}
		f'=\tau'\psi'\sface{p+\epsilon}\nu'\phi',
	\end{equation*}
	where $\delta'=\sface{p+\epsilon}$ is a single standard face, with $\epsilon=1,0$.

	The degenerated face $\dface{p}$ does not interact with $\tau'$, so the second equality of the relation \eqref{rel:dv} is used to move it through the vertical faces of $\tau'$.
	However, while passing through $\psi'$ using \eqref{rel:sw}, $\dface{p}$ may interact with one of the bordering extensions and produce a vertical face (as per the third equality of \eqref{rel:sw}).
	In such cases, the resulting vertical face is moved to $\tau'$ using \eqref{rel:vw} of \cref{prop:ovrel} (the second equality is not applicable) and \cref{lem:reloc}.
	Ultimately, $\dface{p}$ combines with the standard face to give $f=\tau'\psi'\bbface{p}{\epsilon}\nu'\phi'$, which allows us to move $\bbface{p}{\epsilon}$ at the beginning of the composition using the relations \eqref{rel:vw}, \eqref{rel:ww1} and \eqref{rel:ww2} of \cref{prop:ovrel}.
	The number of degenerated, standard and vertical faces in the composite $\tau'\psi'\nu'\phi'$ is less than $n$ and the bottom map is an identity, therefore we can use the inductive hypothesis to reorder the vertical faces and bordering extensions and obtain
	\begin{equation*}
		f=\vface{k_{\gamma}}\cdots\vface{k_{1}}\bbface{c_{\theta}}{\epsilon_{\theta}}\cdots\bbface{c_1}{\epsilon_1}\bbface{p}{\epsilon}.
	\end{equation*}
	We use \cref{lem:reloc} to correctly place $\bbface{p}{\epsilon}$ in the factorisation to have an increasing order.

	Suppose $f=f'\sface{p}$, i.e.\ $g$ is a standard face $\sface{p}$.
	As the composition $\fbot{f}'\face_p$ is the identity, $\fbot{f}'$ is a retraction of $\face_p$, and we either have $\fbot{f}'=\sigma_p$ or $\fbot{f}'=\sigma_{p-1}$.
	Therefore, we can factor $f$ as follows
	\begin{equation*}
	\begin{tikzcd}[row sep=1.2cm, column sep=1.2cm]
		\cdot & \cdot & \cdot & \cdot \\
		\cdot & \cdot & \cdot & \cdot
		\ar[from=1-1, to=1-2, hook, "\face_{\bv{}{p}}"]
		\ar[from=1-2, to=1-3, equal]
		\ar[from=1-3, to=1-4, hook, "\ftop{f}"]
		\ar[from=1-1, to=2-1, twoheadrightarrow, "\fn" description]
		\ar[from=1-2, to=2-2, twoheadrightarrow, "\fm" description]
		\ar[from=1-3, to=2-3, twoheadrightarrow, "\fk\ftop{f}" description]
		\ar[from=1-4, to=2-4, twoheadrightarrow, "\fk" description]
		\ar[from=2-1, to=2-2, hook, "\face_p"']
		\ar[from=2-2, to=2-3, hook, "\sigma_{p-\epsilon}"']
		\ar[from=2-3, to=2-4, equal]
		\ar[from=1-2, to=2-3, phantom, "\dface{p-\epsilon}" description]
		\ar[from=1-1, to=2-2, phantom, "\sface{p}" description]
	\end{tikzcd}
	\end{equation*}
	where $\epsilon=0,1$.
	The left and middle squares compose to give us a bordering extension $\bbface{p-\epsilon}{\epsilon}$.
	In addition, the right square as an identity at the bottom and the number of degenerated, standard and vertical faces factoring it is less than $n$.
	By using the inductive hypothesis, we factor it to obtain
	\begin{equation*}
		f=\vface{k_{\gamma}}\cdots\vface{k_{1}}\bbface{c_{\theta}}{\epsilon_{\theta}}\cdots\bbface{c_1}{\epsilon_1}\bbface{p-\epsilon}{\epsilon}.
	\end{equation*}
	We proceed as above to correctly place the $\bbface{p-\epsilon}{\epsilon}$ in the increasing order.
\end{proof}

\begin{lem}\label{lem:monouniqfact}
    The factorisation of \cref{lem:monofact} is unique up to the relations \eqref{rel:dd}--\eqref{rel:hd}.
\end{lem}
\begin{proof}
	We prove the uniqueness, up to the relations, of the factorisation \eqref{diag:monofact} by induction on the number $c$ of vertical, degenerated and standard faces in the decomposition of $f\:\fn \to \fk$.
	The case $c=0$ is trivial.
    Suppose any factorisation of length $c$ is unique, and assume that $f$ decomposes into $c+1$ vertical, degenerated and standard faces.
	In particular, we have that $f=gf'$ or $f=f'g$, where $g$ is either a degenerated, standard or vertical face.

    Consequently, the map $f'$ has a factorisation of length $c$.
	By hypothesis, the factorisation is unique and thus, can take the form of \cref{lem:monofact}:
    \begin{equation*}
        f'=\tau'\psi'\delta'\nu'\phi'=\vface{t_{\iota}}\cdots\vface{t_1}\bbface{i_{\alpha}}{\epsilon_{\alpha}}\cdots\bbface{i_{1}}{\epsilon_1}\sface{j_{\beta}}\cdots\sface{j_{1}}\vface{k_{\gamma}}\cdots\vface{k_{1}}\bbface{c_{\theta}}{\epsilon_{\theta}}\cdots\bbface{c_1}{\epsilon_1},
    \end{equation*}
	where the faces are arranged in increasing order.
	In this proof, we will always assume that $f=gf'$.

    First, suppose the morphism $g$ is a vertical face $\vface{p}$.
	Two cases can be distinguished: either $\fk(p)$ is in the image of $\fbot{f}$ or it is not.
	In the first case, this means that $\vface{p}$ inserts an inner marked vertex to build up the fibre $\ffib{\fn}{\fk(p)}$.
	The vertical face must travel to $\nu'$ to reach the correct position and, as the maps $\tau'$, $\psi'$ and $\delta'$ do not interact with $\dvp$ (they construct the fibres that are not in the image of $f$), we can use the relations \eqref{rel:vv}, \eqref{rel:hv}, \eqref{rel:dv} and \eqref{rel:vw} to do so and obtain
	\begin{equation*}
        f=\tau'\psi'\delta'\vface{\dvp}\nu'\phi'.
    \end{equation*}
	We use \cref{lem:vfactuni} to conclude this case.
	In the second case, since $\fk(p)$ is not in the image of $\fbot{f}$, the face $\vface{p}$ inserts an inner marked vertex in the fibre $\aint{p}{\fk}$ and hence must be placed in $\tau'$, which is done using \cref{lem:reloc}.

	Suppose $g$ is a degenerated face $\dface{p}$.
	Since the bottom map $\fbot{f}$ is a monomorphism, $\dface{p}$ must combine with a standard face to form a bordering extension.
	The second equality of \eqref{rel:dv} can be utilised to navigate $\dface{p}$ through $\tau'$ since they do not interact.
	In fact, the vertical faces cannot split the edges freshly marked by $\dface{p}$ as it comes at the end of the composition.
	The face $\dface{p}$ can interact with the composition of bordering extensions $\psi'$ if it marks an unmarked edge neighbouring a bordering vertex inserted by one of the bordering extensions in $\psi'$.
	In that case, there is a $1 \leq \alpha_0 \leq \alpha$ such that $p=i_{\alpha_0}-1+\epsilon_{\alpha_0}$ and after using \eqref{rel:sw} of \cref{prop:ovrel} to go through the first part of the bordering extensions and commute $\dface{p}\bbface{i_{\alpha_0}}{\epsilon_{\alpha_0}}$ we obtain
	\begin{equation*}
		\bbface{{i_{\alpha}}}{\epsilon_{\alpha}}\cdots\vface{\bv{}{i_{\alpha_0}+\epsilon_{\alpha_0}}}\dface{p}\cdots\bbface{{i_1}}{\epsilon_{\alpha_1}}.
	\end{equation*}
	We use the relations \eqref{rel:vv} and \eqref{rel:vw} to correctly place $\vface{\bv{}{i_{\alpha_0}+\epsilon_{\alpha_0}}}$ in $\tau'$ and use \cref{prop:ovrel} again to address the remaining bordering extensions with \eqref{rel:sw}.
	In the second case, if $\dface{p}$ does not interact with the $\psi'$ we simply use \eqref{rel:sw} to obtain the same result.
	It remains to combine $\dface{\dvp}$ with a standard face and relocate the bordering extension produced to $\psi'$.
	For that, there is a $1 \leq \beta_0 \leq \beta$ such that $\dvp = j_{\beta_0}-1$.
	We use \eqref{rel:dd} to move $\sface{j_{\beta_0}}$ next to $\dface{\dvp}$ and combine it to obtain the following
	\begin{equation*}
		f=\tau'\psi'\bbface{\dvp}{1}\sface{j_{\beta}}\cdots\widehat{\sface{j_{\beta_0}}}\cdots\sface{j_{1}}\nu'\phi'.
	\end{equation*}
	By \cref{lem:monouniqfact} we are done.

	Suppose the morphism $g$ is a standard face $\sface{p}$.
	As it is not possible to attach a bordering vertex on a standard vertex that has not been created yet, or to insert an inner marked vertex on a fresh standard vertex, $\sface{p}$ does not interact with $\tau'$ and $\psi'$.
	Therefore, we can use \eqref{rel:hv} and the relation \eqref{rel:wd} of \cref{prop:ovrel} to navigate through  $\tau'$ and $\psi'$ and obtain
	\begin{equation*}
		f=\tau'\psi'\sface{p}\delta'\nu'\phi'.
	\end{equation*}
	\cref{lem:reloc} allows us to correctly place $\sface{p}$ in the factorisation of $\delta'$.
\end{proof}

\begin{thm}\label{thm:grdu}
	The factorisation of \cref{thm:genreldec} is unique up to relations \eqref{rel:dd}--\eqref{rel:hd}.
\end{thm}
\begin{proof}
	We prove the uniqueness, up to the relations, of the factorisation by induction on the number $c$ of degenerated, vertical and standard faces in the decomposition of $f$.
	The case $c=0$ is trivial. Suppose any factorisation of length $c$ into degenerated, vertical and standard faces is unique, and assume that $f$ decomposes into $c+1$ faces.
	In particular, we have that $f=gf'$, where $g$ is either a vertical, standard or degenerated face. Consequently, the map $f'$ has a factorisation of length $c$.
	By hypothesis, the factorisation is unique and thus, can take the form \eqref{eq:dvhdec}.
	Suppose $g$ is either a vertical face $\vface{p}$ or a standard face $\sface{p}$.
	By \cref{lem:monouniqfact}, we are done.
	Therefore, suppose $g$ is a degenerated face $\dface{p}$.
	If $\fk(p)$ is not in the image of $\fbot{f}$, then $\dface{p}$ must interact with a standard face from the factorisation of $\delta'$.
	Specifically, the degenerated face $\dface{p}$ has to combine with the appropriate standard face to form a bordering extension.
	In this situation, we once again apply \cref{lem:monouniqfact} to conclude the proof.
	If $\fk(p)$ is in the image of $\fbot{f}$, then $\dface{p}$ does not interact with $\tau'$ and $\psi'$.
	Therefore, we can use \eqref{rel:dv} and the relation \eqref{rel:sw} of \cref{prop:ovrel} and to pass through them and obtain
	\begin{equation*}
		f=\tau'\psi'\dface{p}\delta'\nu'\phi'\sigma'.
	\end{equation*}
	For the interaction with the standard faces, we have to take care of three cases: for $p<j_1-1$, for $p>j_{\beta}$ and for $j_{\beta}\geq p\geq j_{1}-1$.
	In the first and second cases, we can use the third and first equality, respectively, of \eqref{rel:hd} to navigate through $\sface{j_{\beta}}\cdots\sface{j_{1}}$.
	Suppose $j_{\beta}\geq p\geq j_{1}-1$, then there is a maximal $1\leq\beta_0\leq \beta$ such that $p\geq j_{\beta_0}-1$. We use the third equality of \eqref{rel:hd} to obtain
	\begin{equation*}
		\sface{j_{\beta}-1}\cdots\dface{p}\sface{j_{\beta_0}}\cdots\sface{j_{1}}.
	\end{equation*}
	In the case where $p=j_{\beta_0}-1$ or $p=j_{\beta_0}$, we use the second equality of \eqref{rel:hd} to produce the bordering extension $\bbface{p}{\epsilon}$.
	Since $\fk(p)$ is in the image of $\fbot{f}$, we reposition it to $\phi'$ by using the relations \eqref{rel:ww1} and \eqref{rel:ww2} of \cref{prop:ovrel,lem:monouniqfact} to conclude.
	Otherwise, $p>j_{\beta_0}$ and, similarly as in the second case, we use the first equality of \eqref{rel:hd} to navigate through the rest of the standard faces and obtain
	\begin{equation*}
		f=\tau'\psi'\delta'\dface{\dvp}\nu'\phi'\sigma'.
	\end{equation*}
	The relations \eqref{rel:sw} from \cref{prop:ovrel} and \eqref{rel:dv} can be used to traverse the composition of vertical faces $\nu'$ and bordering extensions $\phi'$, and the vertical faces emerging can be repositioned if necessary.
	Finally, for the arrangement of the degenerated faces, we use \cref{lem:reloc}.
\end{proof}

Observe that, from the proofs of \cref{lem:monofact,lem:monouniqfact,thm:genreldec,thm:grdu}, we can recover the ternary factorisation system $(\fDiag,\fVert,\fHorz)$ mentioned in \cref{para:tfact}.

\begin{prop}\label{cor:tfs}
	The classes $(\fDiag,\fVert,\fHorz)$ form a ternary factorisation system on $\fDel$.
\end{prop}
\begin{proof}
	Let $f\:\fn \to \fk$ be a morphism in $\fDel$.
	By \cref{thm:genreldec,thm:grdu}, we can factor $f$ uniquely up to the relations \eqref{rel:dd}--\eqref{rel:hd} as in \eqref{eq:dvhdec}
	\begin{equation*}
		f=\vface{t_{\iota}}\cdots\vface{t_1}\bbface{i_{\alpha}}{\epsilon_{\alpha}}\cdots\bbface{i_{1}}{\epsilon_1}\sface{j_{\beta}}\cdots\sface{j_{1}}\vface{k_{\gamma}}\cdots\vface{k_{1}}\bbface{c_{\theta}}{\epsilon_{\theta}}\cdots\bbface{c_1}{\epsilon_1}\dface{l_{\lambda}}\cdots\dface{l_{1}}.
	\end{equation*}
	Therefore, it suffices to remark that the composition $\vface{t_{\iota}}\cdots\vface{t_1}\bbface{i_{\alpha}}{\epsilon_{\alpha}}\cdots\bbface{i_{1}}{\epsilon_1}\sface{j_{\beta}}\cdots\sface{j_{1}}$ is a horizontal morphism, the composition $\vface{k_{\gamma}}\cdots\vface{k_{1}}\bbface{c_{\theta}}{\epsilon_{\theta}}\cdots\bbface{c_1}{\epsilon_1}$ is a vertical morphism, and the composition $\dface{l_{\lambda}}\cdots\dface{l_{1}}$ is a diagonal morphism.
\end{proof}

In particular, we can say more about the diagonal, vertical and horizontal morphisms.

\begin{cor}
	Any morphism in
	\begin{itemize}
		\item $\fDiag$ can be factored as $\dface{l_{\lambda}}\cdots\dface{l_{1}}$,
		\item $\fVert$ can be factored as $\vface{k_{\gamma}}\cdots\vface{k_{1}}\bbface{c_{\theta}}{\epsilon_{\theta}}\cdots\bbface{c_1}{\epsilon_1}$,
		\item $\fHorz$ can be factored as $\vface{t_{\iota}}\cdots\vface{t_1}\bbface{i_{\alpha}}{\epsilon_{\alpha}}\cdots\bbface{i_{1}}{\epsilon_1}\sface{j_{\beta}}\cdots\sface{j_{1}}$,
	\end{itemize}
	where the indices are in increasing order.\qed
\end{cor}
